%% file: BeckChevalleyFibrations.tex
\documentclass[10pt, notitlepage]{article}
\usepackage{amsmath,amssymb,amsthm} 
\usepackage[linktocpage=true]{hyperref}
\usepackage{cleveref} 
\usepackage[dvipsnames]{xcolor}
\usepackage{geometry} 
\usepackage{lmodern} 

\newcommand*{\boldone}{\text{\usefont{U}{bbold}{m}{n}1}}
\let\mathbbalt\mathbb

\let\mathbb\mathbbalt

\usepackage{eucal} 
\usepackage[english]{babel}
\usepackage{graphicx}
\usepackage{csquotes} 
\usepackage[style=alphabetic, backend=biber]{biblatex} 
\usepackage{tikz-cd}
\usetikzlibrary{babel} 
\tikzcdset{every diagram/.style={sep={large},/tikz/baseline=0pt}}
\usepackage{enumitem} 

\setlist[enumerate]{label=(\arabic*)}

\newcommand{\R}{\mathcal{R}}
\newcommand{\Q}{\mathbb{Q}}

\newcommand{\C}{\mathcal{C}}
\newcommand{\Z}{\mathbb{Z}}
\newcommand{\F}{\mathbb{F}}

\DeclareMathOperator*{\colim}{colim}
\DeclareMathOperator{\Fun}{Fun}
\DeclareMathOperator{\Map}{Map}
\DeclareMathOperator{\id}{id}
\DeclareMathOperator{\Nm}{Nm}
\DeclareMathOperator{\LocSys}{LocSys}
\DeclareMathOperator{\Cart}{Cart}
\DeclareMathOperator{\ev}{ev}
\DeclareMathOperator{\coker}{coker}

\DeclareMathOperator{\Sp}{Sp}
\DeclareMathOperator{\Str}{Str}
\DeclareMathOperator{\Un}{Un}

\renewcommand{\P}{\mathcal{P}}
\renewcommand{\to}{\longrightarrow}

\newcommand{\Ab}{\mathsf{Ab}}

\newcommand{\Grpd}{\mathsf{Grpd}}
\newcommand{\Cat}{\mathsf{Cat}}

\newcommand{\Vect}{\mathsf{Vect}}

\newcommand{\op}{\mathrm{op}}
\newcommand{\X}{\mathcal{X}}
\newcommand{\Y}{\mathcal{Y}}

\let\paragraph\S
\renewcommand{\S}{\mathcal{S}}
\newcommand{\D}{\mathcal{D}}
\newcommand{\E}{\mathcal{E}}

\newcommand{\wt}{\widetilde} 

\title{Beck-Chevalley Fibrations}
\author{Thomas H. Surlykke}
\date{}

\usepackage{soul} 
\usepackage[perpage]{footmisc}

\hypersetup{
    colorlinks,
    linkcolor={red!50!black},
    citecolor={blue!50!black},
    urlcolor={blue!80!black}
}
\theoremstyle{plain}
\newtheorem*{theorem*}{Theorem}
\newtheorem{theorem}{Theorem}[section]
\newtheorem{proposition}[theorem]{Proposition}
\newtheorem{lemma}[theorem]{Lemma}
\newtheorem{corollary}[theorem]{Corollary}
\theoremstyle{definition}
\newtheorem{definition}[theorem]{Definition}
\newtheorem{notation}[theorem]{Notation}
\newtheorem{construction}[theorem]{Construction}
\newtheorem{remark}[theorem]{Remark}
\newtheorem{example}[theorem]{Example}

\definecolor{darkgreen}{RGB}{0, 120, 0}
\renewcommand{\d}{\textcolor{darkgreen}}
\renewcommand{\L}{\mathcal{L}}

\crefname{equation}{diagram}{diagrams}
\Crefname{equation}{Diagram}{Diagrams}
\crefformat{equation}{diagram (#2#1#3)}
\Crefformat{equation}{Diagram (#2#1#3)}

\begin{document}
\maketitle
\linespread{1.05}
\input{abstract.tex}

\tableofcontents

\section{Introduction}
\input{introduction/introduction.tex}
\clearpage

\section{Preliminaries}
\subsection{The Straightening Theorem} \label{sec:straightening}
\input{definitions/straightening.tex}
\subsection{Beck-Chevalley Fibrations and Ambidexterity}
\input{definitions/definitions.tex}

\clearpage

\section{The Norm Square}
\subsection{Beck-Chevalley Transformations}
\input{normSquare/bcMaps.tex}

\subsection{The Norm Square Theorem}
\input{normSquare/normMap.tex}
\subsection{Base Change of Beck-Chevalley Fibrations}
\input{normSquare/baseChange.tex}

\clearpage
\section{Applications}
\input{applications/normApplications.tex}

\emergencystretch=1em
\clearpage
\printbibliography[heading=bibintoc]
\end{document}

%% file: abstract.tex
\begin{abstract}
  We extend the theory of ambidexterity developed by M.J. Hopkins and J. Lurie
  by proving commutativity of the \emph{norm square} induced from a weakly
  ambidextrous morphism by two Beck-Chevalley fibrations that are associated by a
  functor. By showing how ambidexterity is preserved under base change of
  Beck-Chevalley fibrations, we demonstrate that our result is a generalization
  of the naturality property of the norm shown by M.J. Hopkins and J. Lurie.
  Furthermore, we demonstrate how our generalization implies two specific
  results previously shown by S. Carmeli, T. M. Schlank, and L. Yanovski,
  namely, that the induced norm square of local systems, and the induced norm
  square of equivariant powers, both commute.
\end{abstract}

%% file: introduction/introduction.tex
\subsection{Background}
\subsubsection*{Representation Theory}
Given a $1$-category $\C$, we can identify an action of a group $G$ on an
object $X \in \C$ with a functor $BG \to \C$ from the \emph{classifying space}
of $G$ to $\C$, which sends the unique object of $BG$ to $X$. Thus, the
category $\Fun(BG, \C)$ is the category of objects of $\C$ equipped with a
$G$-action. Consider the case where $\C$ is the category $\Vect_{k}$ of vector
spaces over a field $k$. Given a vector space $V$ with a $G$-action, two vector
spaces related to $V$ are of interest: the subspace of \emph{$G$-invariants} of
$V$,
\[
  V^{G} = \{ x \in V \mid \forall g \in G : gx = x \},
\]
and the quotient vector space of \emph{$G$-coinvariants}, $V_{G}$, which is
defined as the quotient of $V$ by the subspace generated by elements of the
form $x - gx$ for $x \in V$ and $g \in G$. If $G$ is a finite group, there is a
well-defined natural map
\[
  \Nm_{G} \colon V_{G} \to V^{G}
\]
from the coinvariants to the invariants, given by $x \mapsto \sum_{g \in G}
gx$. We refer to this map as the \emph{norm map} associated to $V$. If $k$ has
characteristic 0, such as $\Q$, then $\Nm_{G}$ is an equivalence. However, in
general, this norm map is not an equivalence.

To make this construction functorial, consider the functor $\Fun(-, \C) \colon
\Grpd^{\op} \to \Cat$, which takes a groupoid $X$ to the functor category
$\Fun(X, \C)$. Applying this functor to the projection $p \colon BG \to *$, we
obtain the precomposition functor $p^{*} \colon \C \to \Fun(BG, \C)$, which
sends an object of $\C$ to the same object, equipped with the trivial action of
$G$. If $\C$ admits limits and colimits, then $p^{*}$ admits a left adjoint
$p_{!}$ and a right adjoint $p_{*}$ by left and right Kan extension,
respectively. As Kan extensions along the projection to a point are given by
computing the (co)limit, we for each $X \in \Fun(BG, \C)$ have $p_{!}(X) =
\colim_{BG} X$ and $p_{*}(X) = \lim_{BG} X$. In the case where $\C =
\Vect_{k}$, we for a vector space $V$ equipped with a $G$-action obtain
$p_{!}(X) = V_{G}$ and $p^{*}(X) = V^{G}$.

This construction relates to \emph{Tate cohomology}. If $\C$ is an abelian
category, then $p_{!}$ is a \emph{right exact functor} and $p_{*}$ is a
\emph{left exact functor} as a consequence of these functors being left and
right adjoints respectively. The functors are not in general exact, so it is
natural to consider the \emph{right derived functors} of $p_{!}$
\[
  H_i(G; -) \colon \Fun(BG, \C) \to \C,
\]
and the \emph{left derived functors} of $p_{*}$
\[
  H^{i}(G; -) \colon \Fun(BG, \C) \to \C,
\]
for each integer $i\geq 0$. Letting $\Ab$ denote the category of abelian groups,
note that for a group $G$, the functor category $\Fun(BG, \Ab)$ is canonically
equivalent to the category of $G$-modules. Thus, for $\C=\Ab$, this construction
provides the definitions of \emph{group homology} and \emph{group cohomology}.
Now, as before, let $\C$ be the category $\Vect_{k}$ of $k$-vector spaces. A
vector space with a $G$ action is in particular a $G$-module, so the derived
functors also define group (co)homology in this situation. As for any
$G$-module $V$ the $0$th derived functors are given by $H_{0}(G; V) \cong
V_{G}$ and $H^{0}(G;V) \cong V^{G}$, we see as above that we have a norm map
$\Nm_{G} \colon H_{0}(G; V) \to H^{0}(G; V)$ when $G$ is a finite group. This
relation between the $0$th homology and cohomology groups leads to the
definition of \emph{Tate cohomology}, which for any finite group $G$ is given
by
\[
  \widehat{H}^{i}(G; V) =
  \begin{cases}
    H^{n}(G; V)     &   n \geq 1 \\
    \coker \Nm_{G}  &   n=0 \\
    \ker \Nm_{G}    &   n=-1  \\
    H_{-(n+1)}(G;V) &   n\leq -2
  \end{cases}.
\]
In particular, the Tate cohomology groups measure the failure of the norm map
$\Nm_{G}$ being an equivalence.

All of the above constructions also apply in the $\infty$-categorical setting.
Let $\C$ be an $\infty$-category, and consider the functor $\Fun(-, \C) \colon
\S^{\op} \to \Cat_{\infty}$ from the \emph{$\infty$-category of spaces} $\S$ to
the \emph{$\infty$-category of $\infty$-categories}. Since the classifying
space $BG$ of a group $G$ is a Kan complex, we can again consider the image of
the projection $p \colon BG \to *$, which is the precomposition functor $p^{*}
\colon \C \to \Fun(BG, \C)$. As before, the left adjoint $p_{!}$ and right
adjoint $p_{*}$ of $p^{*}$, when they exist, are given by computing the colimit
and limit, respectively. Thus, we obtain $p_{!}(X) = X_{hG}$, the
\emph{homotopy orbits} of $X$, and $p_{*}(X)= X^{hG}$, the \emph{homotopy fixed
points} of $X$. When $G$ is finite, we also have a \emph{norm map} $\Nm_{G}
\colon X_{hG} \to X^{hG}$.

In this paper, we define the notion of \emph{(weak) ambidexterity} of morphisms
in a given base category, with respect to certain contravariant functors into
$\Cat_{\infty}$. When the base category is $\S$, the $\infty$-category of
spaces, and the functor is $\Fun(-, \C)$ for some $\infty$-category $\C$ which
admit limits and colimits, these definitions have the following properties: If
$f$ in $\S$ is \emph{weakly ambidextrous}, then there exists a natural
transformation $\Nm_{f} \colon f_{!} \to f_{*}$ which we call the \emph{norm
map} associated to $f$. Moreover, if $f$ is \emph{ambidextrous}, then this norm map
is an equivalence. Therefore, the question of whether $\Nm_{G} \colon V_{G} \to
V^{G}$ exists is a question of whether $p \colon BG \to *$ is weakly
ambidextrous, and the question of whether $\Nm_{G} \colon V_{G} \to V^{G}$ is
an equivalence is a question of whether $p$ is ambidextrous. In
general, for $X$ a space, we refer to an object in the $\infty$-category
$\Fun(X, \C)$ as a \emph{$\C$-valued local system on $X$}. It was shown by
Yonatan Harpaz in \cite{harpaz} that ambidexterity for local systems $\Fun(-,
\C)$ can be encoded as the action of a span category.

\subsubsection*{Chromatic Homotopy Theory}
In the following, denote the $\infty$-category of spectra by $\Sp$. Another
field of mathematics in which ambidextrous maps show up is \emph{chromatic
homotopy theory}.
Chromatic homotopy theory relates the height filtration of formal groups with
the height filtration of spectra. Consider the complex cobordism spectrum $MU$.
On the one hand, as a spectrum $MU$ represents a cohomology theory -- in fact,
$MU$ represents a complex-oriented cohomology theory. On the other, $MU$
is also universal among formal groups: The graded ring $\pi_{*}MU$ is
equivalent to the \emph{Lazard ring} $L$. The Lazard ring is universal among
formal groups in the following sense: A formal group law $f(x,y) \in R[[x]]$
over $R$ can be written as a sum $f(x,y) = \sum_{i,j\geq 0}c_{ij}x^{i}y^{j}$,
where the coefficients $c_{ij}$ has to satisfy certain conditions. The Lazard
ring $L$ is the free ring on these coefficients $\Z[c_{ij}]$ modulo the
relations the coefficients have to satisfy. Hence, picking out coefficients of
$R$ that amount to a formal group law is equivalent to picking a ring
homomorphism $L \to R$. It was shown by Quillen that $\pi_{*}MU$ is isomorphic
as a graded ring to the Lazard ring $L$ (see \cite[Lecture 7, Theorem
1]{chromaticLurie}), and with \cite[Lecture 2, Theorem 4]{chromaticLurie} it is
shown that $L$ is isomorphic to the polynomial ring $\Z[t_1, t_2, \ldots]$ in
infinitely many generators, where each $t_{i}$ has degree $2i$. In particular,
we see that forming a formal group over $R$ is the same as picking a countable
sequence of elements in $R$.

Given a formal group law $f(x,y) \in R[[x]]$, we define the \emph{height} of $f$
(for a given prime $p$) as follows: First, the \emph{$n$-series} $[n](t) \in
R[[t]]$ of $f$ is defined recursively for all $n\geq 0$ by $[0](t) = 0$ and
$[n+1](t) = f([n](t), t)$. Considering the $p$-series $[p](t)$ for a prime $p$,
write $v_{n}$ for the coefficient of $t^{p^{n}}$. As seen in \cite[Lecture 12,
Definition 13]{chromaticLurie}, we say that $f$ has \emph{height $\geq n$} if
$v_{i} = 0$ for $i<n$, and we say that $f$ has \emph{height $\leq n$} if $v_{n}
\in R$ is invertible. Thus, $f$ has \emph{height exactly $n$} if $v_{n}$ is
invertible, and $v_{i}$ vanishes for $i<n$. By the isomorphism $\pi_{*}(MU) \cong
L$, we can equivalently think of these coefficients $v_{n}$ as a certain
sequence of elements $v_{n} \in \pi_{2(p^{n}-1)}(MU)$, for each $n\geq 0$. For
any prime $p$ there exists a sequence of complex oriented cohomology theories,
$K(n)$ for each $n\geq 0$, called the \emph{Morava $K$-theories}. Each $K(n)$
is an $MU$-algebra, and the stable homotopy groups of these spectra are given
by $\pi_{*}(K(n)) = \F_{p}[v_{n}^{\pm 1}]$, where $v_{n}$ is of degree
$2(p^{n}-1)$. Comparing this generator to the height of a formal group, we see
that we should think of the $K(n)$-s as being concentrated exactly at height
$n$.

\emph{Telescopic localization} gives another candidate for localization, which
is localization at a spectrum given as the telescope on a $v_{n}$-self map of a
type $n$ finite spectrum.
For a fixed $n$, letting $T(n)$ be one such spectrum it turns out that the
localization at $T(n)$ is independent of the choices presented, so we obtain a
well-defined $\infty$-category $\Sp_{T(n)}$ of \emph{$T(n)$-local spectra}.
Letting $\Sp_{K(n)}$ denote the $\infty$-category of \emph{$K(n)$-local
spectra}, there is an inclusion of $\infty$-categories, $\Sp_{K(n)} \subseteq
\Sp_{T(n)}$. This inclusion is known to be an equivalence when $n$ is $0$ or
$1$. The question whether this inclusion is strict for larger $n$ is known as
the \emph{Telescope Conjecture}. Localizing $\pi_{*}MU$ at $K(n)$ or $T(n)$
can both be thought of as completing with respect to $v_{0}, \ldots, v_{n-1}$
and inverting $v_{n}$. As is shown in \cite{1811}, it turns out that
$\Sp_{K(n)}$ and $\Sp_{T(n)}$ are precisely the minimal and maximal examples,
respectively, of such localizations.

Let $X \in \Sp$ be a spectrum, and denote  for each $n\geq 0$ the Bousfield
localization of $X$ at $K(0) \vee \cdots \vee K(n)$ by $L_{n} X$. Thus, we
obtain a diagram
\[
  \begin{tikzcd}
    & & X \arrow[dl, ""] \arrow[d, ""] \arrow[drr, ""]& \\
    \cdots \arrow[r, ""]
    & L_{n} \arrow[r, ""]
    & L_{n-1}X \arrow[r, ""]
    & \cdots \arrow[r, ""]
    & L_0 X
  \end{tikzcd}
\]
of maps of spectra, which induces a map from $X$ to the limit of the
$L_{n}$-localizations. The \emph{Chromatic Convergence Theorem}, \cite[Lecture
32, Theorem 1]{chromaticLurie} states, that if
$X$ is a finite spectrum, then the limit of the $L_{n}$-localizations is the
$p$-localization of $X$. That is, the diagram above induces a map
\[
  X \to X_{(p)}.
\]
In particular, if $X$ is a finite $p$-local spectrum, then the limit of the
$L_{n}$-localizations induce an equivalence $X \simeq X_{(p)}$. This, combined
with the fact that for each integer $n\geq 1$, there exists a commuting square
\[
  \begin{tikzcd}
    L_{n} X	\arrow[r, ""] \arrow[d, ""]
    &L_{K(n)}X \arrow[d,""]
    \\
    L_{n-1}X	\arrow[r, ""]
    &L_{n-1}L_{K(n)}X
  \end{tikzcd},
\]
means that, in theory, a finite $p$-local spectrum $X$ can be computed by its
intermediate $K(n)$-localizations.
In practice, these computations are often not easy to do, but the theoretical
value of height localizations still stands. In \cite{ambidexLurie}, Hopkins and
Lurie prove the following result:

\begin{theorem*}[Hopkins-Lurie]
  Let $A$ be a Kan complex. Assume that for every vertex $x \in A$, the sets
  $\pi_{n}(A, x)$ are finite for every integer $n$ and trivial for $n \gg 0$.
  Let $\rho \colon A \to \Sp_{K(n)}$ be a diagram of $K(n)$-local spectra,
  indexed by $A$. In this situation, there is a canonical equivalence
  \[
    \Nm_{A} \colon \colim_{A}\rho \xrightarrow{\simeq} \lim_{A}\rho.
  \]
\end{theorem*}

Recalling the discussion above, by considering the functor $\Fun(-, \Sp_{K(n)})
\colon \S^{\op} \to \Cat_{\infty}$, this statement can be reformulated as
follows: \emph{For $A$ a truncated Kan complex with all homotopy groups finite,
  the projection $p \colon A \to *$ in $\S$ is ambidextrous with respect to the
functor $\Fun(-, \Sp_{K(n)})$}. By letting $A = BG$, we obtain the following
result, also stated in \cite{ambidexLurie}, which was proven separately
(and previously) by Greenlees, Hovey, and Sadofsky:

\begin{theorem*}[Greenlees-Hovey-Sadofsky]
  Let $G$ be a finite group and let $X \in \Sp_{K(n)}$ be a $K(n)$-local
  spectrum equipped with an action of $G$. In this situation, the
  $K(n)$-localized norm map
  \[
    \Nm_{G} \colon X_{hG} \to X^{hG}
  \]
  is an equivalence in the $\infty$-category of $K(n)$-local spectra
  $\Sp_{K(n)}$.
\end{theorem*}

\subsection{Outline}
In this paper, we define and introduce the concept of (weak) ambidexterity of
morphisms with respect to certain contravariant functors into the $\infty$-category of
$\infty$-categories.

In Section 2, we begin by stating Lurie's straightening theorem \cite[Theorem
3.2.0.1]{htt} and its key properties. This allows us to study Cartesian
fibrations $\C \to \X$, in place of functors $\X^{\op} \to \Cat_{\infty}$. We
then introduce the relevant definitions for our paper, namely
\emph{Beck-Chevalley fibrations}, \emph{(weak) ambidexterity}, and the
\emph{norm map} induced by a weakly ambidextrous map, and we state and
prove essential properties relating to these definitions.

In Section 3 we prove the main theorem, \Cref{thm:norm.square}.
Following this, we show how the base change of Beck-Chevalley fibrations along
pullback-preserving functors produces Beck-Chevalley fibrations, and how
ambidexterity is preserved under this base change.

In Section 4 we demonstrate applications of \Cref{thm:norm.square}. First, we
show how it implies \cite[Proposition 4.2.1 (1)]{ambidexLurie}, upon which the
proof of \Cref{thm:norm.square} is based. Following this, we show how
it implies \cite[Theorem 3.2.3]{1811} and  \cite[Proposition 3.4.7]{1811}.

\subsection{Conventions}
We work with $\infty$-categories as \emph{quasi categories}, using the theory
developed by Lurie in \cite{htt} and \cite{ha}. Let $\C$ be an
$\infty$-category. We make the following conventions:
\begin{itemize}
  \item We say that a diagram in $\C$ \emph{commutes} if the diagram commutes
    strictly in the homotopy category $h\C$. This is what is referred to as
    ``commutes up to homotopy'' in \cite{htt} and \cite{ambidexLurie}.

  \item We say that a morphsim in $\C$ is an \emph{equivalence} if it
  is an isomorphism in the homotopy category $h\C$. This is what is referred to
  as a ``homotopy equivalence'' in \cite{ambidexLurie}. We denote an
  equivalence by the symbol $\simeq$.

  \item Given a map $q \colon \C \to \X$ of $\infty$-categories, we refer to
    the pullback $q^{-1}\{X\} = \C \times_{\X} \{X\}$ as the \emph{fiber of $q$
    over $X$}, rather than the homotopy fiber. When $q$ is clear from the
    context, we denote the fiber of $q$ over $X$ by $\C_{X}$.

  \item We let $\S$ denote the \emph{$\infty$-category of spaces}, and let
    $\Cat_{\infty}$ denote the $(\infty,2)$-category of $\infty$-categories.
\end{itemize}

\subsection{Acknowledgements}
Thank you to Shachar Carmeli for suggesting and supervising this project. You
have been very helpful, and it has been a joy working with you.

%% file: definitions/straightening.tex
A fundamental result in higher category theory is Lurie's straightening
theorem, \cite[Theorem 3.2.0.1]{htt}. We first state the following definitions
from \cite{htt}:
\begin{definition}[{\cite[Remark 2.4.1.4]{htt}}]
  Let $q \colon X \to S$ be an inner fibration of simplicial sets, and let $f
  \colon \Delta^1 \to X$ be an edge in $X$. We say that $f$ is
  \emph{$q$-Cartesian} if for every $n\geq 2$ and every commutative diagram
  \[
    \begin{tikzcd}
      \Delta^{\{n-1, n\}} \arrow[dr, "f"] \arrow[d, hookrightarrow, ""]
      & \\
      \Lambda_{n}^{n} \arrow[r, ""] \arrow[d, hookrightarrow, ""]
      & X \arrow[d, "q"]
      \\ \Delta^{n} \arrow[r, ""] \arrow[ur, dashed, ""] & S
    \end{tikzcd}
  \]
  there exists a dashed arrow as indicated, rendering the diagram commutative.
  Dually, we say that $f$ is \emph{$q$-coCartesian} if $f$ is Cartesian with
  respect to the opposite map $q^{\op} \colon X^{\op} \to S^{\op}$.
\end{definition}

\begin{remark}[]
  Assume that $X$ and $S$ are $\infty$-categories, and let a morphism $f$ in
  $X$ be $q$-Cartesian. In the case $n=2$, the property above can be stated as
  follows: If $g$ is a morphism in $X$ with the same target as $f$, and $q(g)$
  factors through $q(f)$ in $S$, then this factorization lifts to a unique
  factorization in $X$, up to equivalence.
\end{remark}

\begin{definition}[{\cite[Definition 2.4.2.1]{htt}}]
  Let $q \colon X \to S$ be a map of simplicial sets. We say that $q$ is a
  \emph{Cartesian fibration} if the following conditions are satisfied:
  \begin{enumerate}
    \item The map $q$ is an inner fibration
    \item For every edge $f \colon x \to y$ in $S$ and every vertex $\wt y$ of
      $X$ with $q(\wt y) = y$, there exists a $q$-Cartesian edge $\tilde{f}
      \colon \wt x \to \wt y$ with $q(\tilde{f}) = f$.
  \end{enumerate}
  We say that $q$ is a \emph{coCartesian fibration} if the opposite map
  $q^{\op} \colon X^{\op} \to S^{\op}$ is a Cartesian fibration. If $q$ is both
  a Cartesian fibration and a coCartesian fibration we will abbreviate and say
  that $q$ is a \emph{Cartesian and coCartesian fibration}.
\end{definition}

Let $q \colon \C \to \X$ be a Cartesian fibration of $\infty$-categories,
viewing $q$ as a map of simplicial sets. Condition (1) implies that for any
object $X \in \X$, the fiber $\C_{X}$ is an $\infty$-category. Condition (2) is
to make sure that this assignment is functorial. Informally, as seen in
\cite{straightening}, the straightening theorem states that for any
$\infty$-category $\X$, there is a \emph{straightening equivalence}, which is a
pair of mutually inverse functors
\[
  \begin{tikzcd}
    \operatorname{Str} : \Cart (\X) \arrow[r, shift left, ""]
    & \arrow[l, shift left, ""] \Fun(\X^{\op}, \Cat_{\infty}) :
    \operatorname{Un}
  \end{tikzcd},
\]
where $\Cart(\X)$ is the category of Cartesian fibrations over $\X$. Applying
$\Str$ is refered to as \emph{straightening} a Cartesian fibration, and
applying $\Un$ is refered to as \emph{unstraightening} a functor of
$\infty$-categories. In general, a functor of $\infty$-categories includes a
lot of coherence data, making functors of $\infty$-categories hard to construct
explicitly. The straightening theorem allows one to define functors between
$\infty$-categories by instead considering certain maps of simplicial sets. We
list some essential properties of the straightning equivalence:
\begin{itemize}
  \item The straightening of a Cartesian fibration $q \colon \C \to \X$ is the
    functor $\C_{(-)} \colon \X^{\op} \to \Cat_{\infty}$. On objects $\C_{(-)}$
    is given by the fiber of $q$, that is,
    \[
      X \longmapsto \C_{X}.
    \]
    On morphisms, given any map $f \colon Y \to X$ in $\X$ a functor $f^{*}
    \colon \C_{X} \to \C_{Y}$ is induced as follows: For every $\wt X \in
    \C_{X}$,  there exists a $q$-Cartesian edge $\tilde{f} \colon \wt X \to \wt
    Y$ in $\C$, since $q$ is a Cartesian fibration. Thus, $f^{*}$ is defined on
    objects by $f^{*}(\wt X) = \wt Y$, and this definition extends to higher
    simplicies as well. The fact that $\tilde{f}$ is $q$-Cartesian ensures
    that this definition is well-defined. In the other direction, the
    unstraightening of a functor $F \colon \X^{\op} \to \Cat_{\infty}$ is an
    $\infty$-categorical analog of the category of elements, that is, the
    \emph{Grothendieck construction}. Thus, informally, the domain of the
    resulting Cartesian fibration $\operatorname{Un}(F) \to \X$ has objects
    pairs $(X, x)$  where $X$ is an object in $\X$ and $x$ is an object in
    $F(X)$.

  \item Suppose that a Cartesian fibration $q \colon \C \to \X$ is also a
    coCartesian fibration. With a similar procedure, $q$ induces a covariant
    functor $\X \to \Cat_{\infty}$ which on objects again is given by $\X
    \longmapsto \C_{X}$, but sends morphisms $f \colon Y \to X$ in $\X$ to
    the lift $f_{!} \colon \C_{Y} \to \C_{X}$ instead.
    The functor $f_{!}$ has the property that it is a \emph{left adjoint}
    to $f^{*}$ (in the sense of \cite[Definition 5.2.2.1]{htt}). More
    generally, the straightening equivalence states that a functor $q \colon \C
    \to \X$ which is \emph{both} a Cartesian fibration and a coCartesian
    fibration is equivalent to giving a functor $\X \to \Cat_{\infty}$ which is
    both covariant and contravariant, and such that for any morphism $f$ in
    $\X$, the covariant image $f_{!}$ of $f$ is left adjoint to the
    contravariant image $f^{*}$ (see \cite[Corollary 5.2.2.5]{htt}).

  \item Suppose we have a diagram
    \[
      \begin{tikzcd}
        & \C \arrow[d, "q"] \\
        \Y \arrow[r, "\alpha"] & \X
      \end{tikzcd}
    \]
    with $q$ a Cartesian fibration. Pulling back along $\alpha$ yields another
    Cartesian fibration $p \colon \C_{\alpha}  \to \Y$, with the property that
    there is a canonical equivalence $(\C_{\alpha})_{Y} \simeq \C_{\alpha(Y)}$
    of fibers for any $Y \in \Y$, as witnessed by the pasting of pullbacks
\[
  \begin{tikzcd}
    (\C_{\alpha})_{Y} \arrow[d] \arrow[r, ""]
    & \C_{\alpha} \arrow[d, "p"] \arrow[r, ""]
    &  \C \arrow[d, "q"] \\
    \{ Y \} \arrow[r, hook] &\Y \arrow[r, "\alpha"]& \X
  \end{tikzcd},
\]
combined with the pasting law for pullbacks, namely the dual of
\cite[Lemma 4.4.2.1]{htt}. Likewise, straightening the Cartesian fibration $q$
lets us consider the diagram
    \[
      \begin{tikzcd}
        \Y^{\op} \arrow[r, "\alpha^{\op}"] & \X^{\op} \arrow[r, "\C_{(-)}"] &
        \Cat_{\infty}
      \end{tikzcd},
    \]
    and we see that precomposition with $\alpha^{\op}$ has the same property of
    indentifying the functor from $\Y^{\op}$ with the functor from $\X^{\op}$
    restricted to the image of $\alpha$. The straightening equivalence states
    this correspondence. Furthermore, if $q$ is also coCartesian fibration,
    then so is any pullback of $q$, as the image of $\C_{(-)} \circ
    \alpha^{\op}$ is contained in the image of $\C_{(-)}$.

  \item Suppose we have a commuting diagram of the form
    \[
      \begin{tikzcd}[column sep=normal]
        \C \arrow[rr, "F"] \arrow[dr, "q"'] && \D \arrow[dl, "p"] \\
                                           & \X &
      \end{tikzcd},
    \]
    with $q$ and $p$ Cartesian fibrations. By commutativity, $F$ induces a
    functor $F_{X} \colon \C_{X} \to \D_{X}$ for each $X \in \X$. If we further
    require that $F$ preserves Cartesian morphisms in $\C$, then the $F$
    determines a natural transformation $F \colon \C_{(-)} \Longrightarrow
    \D_{(-)}$ between functors in $\Fun(\X^{\op}, \Cat_{\infty})$. In other
    words, giving a functor $F \colon \C \to \D$ which preserves Cartesian
    edges is equivalent to giving a functor $F \colon \X^{\op} \times \Delta^{1} \to
    \Cat_{\infty}$ such that $F$ restricted to $\X^{\op} \times \{0\}$ is $\C_{(-)}$,
    $F$ restricted to $\X^{\op} \times \{1\}$ is $\D_{(-)}$ and $F$ restricted to
    $\{X\} \times \Delta^{1}$ is $F_{X}  \colon \C_{X} \to \D_{X}$ for any $X
    \in \X$.
\end{itemize}
We state a reformulation of Construction 4.3.1 in \cite{ambidexLurie}.
\begin{example}[] \label{thm:local.systems}
  Let $\C$ be an $\infty$-category. The assignment $A \mapsto \Fun(A, \C)$
  determines a functor $\Fun(-, \C) \colon \S^{\op} \to \Cat_{\infty}$. Thus,
  by unstraightening we obtain a Cartesian fibration over $\S$. We denote this
  Cartesian fibration by $q \colon \LocSys(\C) \to \S$, and refer to the domain
  $\LocSys(\C)$ as the \emph{$\infty$-category of $\C$-valued local systems}.
  Informally, the objects of $\LocSys(\C)$ are pairs $(A, \L)$, where $A$ is a
  space, and $\L$ is an object in $\Fun(A, \C)$, i.e. a
  \emph{$\C$-valued local system on $A$}. There is a canonical equivalence
  $\LocSys(\C)_{A} \simeq \Fun(A, \C)$ for any $A \in \S$, and for any morphism
  $f \colon B \to A$ in $\S$, the induced functor $f^{*} \colon \LocSys(\C)_{A}
  \to \LocSys(\C)_{B}$ can be canonically identified with the precomposition
  functor $f^{*} \colon \Fun(A, \C) \to \Fun(B, \C)$.

  Let $f \colon B \to A$ be a morphism in $\S$. If $\C$ admits colimits, then
  $f^{*} \colon \Fun(A, \C) \to \Fun(B, \C)$ admits a left adjoint $f_{!}$ by
  left Kan extension. It follows that if $\C$ admits colimits, then the
  Cartesian fibration $q \colon \LocSys(\C) \to \S$ is also a coCartesian
  fibration.
\end{example}

%% file: definitions/definitions.tex
We state the relevant definitions and results, most being taken from
\cite{ambidexLurie}.

\begin{definition}[]
    As in \cite{1811}, we define a \emph{left adjoint} to a functor $F \colon
    \C \to \D$ between $\infty$-categories to be a pair $(L, \eta)$ where $L
    \colon \D \to \C$ is a functor, and $\eta \colon
    \id_{\D} \to FL$ is a \emph{unit transformation} in the sense of
    \cite[Definition 5.2.2.7]{htt}. Similarly, a \emph{right adjoint} to $F$ is
    a pair $(R, \varepsilon)$ where $R \colon \D \to \C$ is a functor, and
    $\varepsilon \colon \id_{\C} \to RF$ is a \emph{counit transformation},
    i.e. the dual of \cite[Definition 5.2.2.7]{htt}. \cite[Proposition
    5.2.2.8]{htt} shows that this definition is equivalent to the definition of
    adjunctions given by \cite[Definition 5.2.2.1]{htt}.
\end{definition}

\begin{remark}[]
  By \cite[Proposition 5.2.2.9]{htt}, any adjunction $L \dashv R$ of functors
  between $\infty$-categories induces an adjunction of $1$-categories when passing
  to the homotopy categories. Furthermore, the unit and counit transformations
  are sent to the respective units and counits of the adjunction of
  $1$-categories. It follows that anytime we have an adjunction $L \dashv R$
  between functors of $\infty$-categories, with unit transformation $\eta
  \colon \id \to RL$ and counit transformation $\varepsilon \colon LR \to \id$,
  these natural transformations satisfy \emph{the triangle identities}. That
  is, the diagrams
  \begin{align*}
    &\begin{tikzcd}[ampersand replacement=\&]
      L \arrow[r, "\varepsilon_{L}"] \arrow[dr, equals, ""]\& LRL \arrow[d, "L \eta"] \\
                                     \& L
    \end{tikzcd} && \text{and} &&
    \begin{tikzcd}[ampersand replacement=\&]
      R \arrow[r, "R \varepsilon"] \arrow[dr, equals, ""] \& RLR \arrow[d, "\eta_{R}"] \\
                                   \& R
    \end{tikzcd}
  \end{align*}
  commute.
\end{remark}

\begin{notation}
  Let $\C$ and $\X$ be $\infty$-categories and let $q \colon \C \to \X$ be a
  Cartesian fibration. Given a morphism $f \colon Y \to X$, we will refer to
  the induced functor of fibers as the \emph{Carteisan lift of $f$ by $q$} and
  denote it by $f^{*} \colon \C_{X} \to \C_{Y}$. In any situation where $f^{*}$
  admits a \emph{left adjoint} we will denote this functor by $f_{!} \colon
  \C_{Y} \to \C_{X}$. Likewise, if $f^{*}$ admits a \emph{right adjoint} we
  will denote this functor by $f_{*} \colon \C_{Y} \to \C_{X}$. If $q$ is also
  a coCartesian fibration, then $f^{*}$ always admits a left adjoint $f_{!}$,
  as mentioned in \Cref{sec:straightening}. In this situation, we will refer to
  $f_{!}$ as the \emph{coCartesian lift of $f$ by $q$}.
\end{notation}
\begin{notation} \label{thm:unit.counit}
  Let $\C$ and $\X$ be $\infty$-categories and let $q \colon \C \to \X$ be a
  Cartesian and coCartesian fibration. As described in
  \Cref{sec:straightening}, for any morphism $f \colon Y \to X$ in $\X$ we
  induce an adjunction
  \[
    \begin{tikzcd}
      \C_{Y} \arrow[r, shift left, "f_{!}"] & \C_{X} \arrow[l, shift left, "f^{*}"]
    \end{tikzcd}.
  \]
  We denote the \emph{unit} and \emph{conuit} of this adjunction by
  \[
    \eta_{f} \colon \id_{\C_{Y}} \to f^{*}f_{!} \qquad \text{and} \qquad
    \varepsilon_{f} \colon f_{!}f^{*} \to \id_{\C_{X}}
  \]
  respectively.

  Suppose that $f^{*}$ admits a right adjoint $f_{*}$ in addition to the left
  adjoint $f_{!}$. In this situation, we denote the \emph{unit} and
  \emph{conuit} of the adjunction $f^{*} \dashv f_{*}$ by
 \[
    \eta_{*f} \colon \id_{\C_{Y}} \to f^{*}f_{*} \qquad \text{and} \qquad
    \varepsilon_{*f} \colon f_{*}f^{*} \to \id_{\C_{X}}
  \]
  respectively.
\end{notation}

\begin{definition} \label{thm:bc.maps}
  Let $\sigma$ be a commutative square of functors between $\infty$-categories,
  depicted as
\[
  \begin{tikzcd}[ampersand replacement = \&]
    \mathcal{F} \arrow[dr, phantom, "\sigma"]  	\&		\D \arrow[l,
    "i^{*}"']  \\ \mathcal{E}	\arrow[u, "h^{*}"'] \&	 \C \arrow[l, "f^{*}"']
    \arrow[u, "g^{*}"']
  \end{tikzcd}.
\]
 By commutativity of the diagram, we obtain a canonical natural equivalence
 $h^{*}f^{*} \simeq i^{*}g^{*}$. Assume that the functors $f^{*}$ and $i^{*}$
 admit left adjoints $f_{!}$ and $i_{!}$ respectively. Denote by $BC[\sigma]$
 the composite natural transformation
\[
  BC[\sigma] \colon i_{!} h^{*} \xrightarrow{i_{!}h^{*}\eta_{f}} i_{!} h^{*}f^{*}f_{!}
  \simeq i_{!}i^{*}g^{*}f_{!} \xrightarrow{\varepsilon_{i}g^{*}f_{!}} g^{*}f_{!}.
\]
We refer to $BC[\sigma]$ as the \emph{Beck-Chevalley transformation} associated
to $\sigma$. If the natural transformation $BC[\sigma]$ is an equivalence, we
will say that $\sigma$ satisfies the \emph{Beck-Chevalley condition}.
\end{definition}

\begin{remark}[]
  In the situation of \Cref{thm:bc.maps}, suppose instead that $f^{*}$ and
  $i^{*}$ admit right adjoints $f_{*}$ and $i_{*}$ respectively. In this
  situation, there is a Beck-Chevalley transformation of right adjoints, given
  by the composite natural transformation
  \[
      BC_{*}[\sigma] \colon g^{*} f_{*} \xrightarrow{\eta_{*i}g^{*}f_{*}}
    i_*i^{*}g^{*}f_{*} \simeq i_*h^{*}f^{*}f_{*} \xrightarrow{i_{*}h^{*}\varepsilon_{*f}}
    i_{*}h^{*}.
  \]
  If there is a need to distinguish between these we will denote the
  Beck-Chevalley transformation of left adjoints by $BC_{!}[\sigma]$ and the
  Beck-Chevalley transformation of right adjoints by $BC_{*}[\sigma]$. If we
  write $BC[\sigma]$ with no index, this is always meant as the Beck-Chevalley
  transformation of left adjoints, that is $BC[\sigma] = BC_{!}[\sigma]$. Note
  that $\sigma$ satisfying the Beck-Chevalley condition will always mean that
  $BC_{!}[\sigma]$ is an equivalence.
\end{remark}

\begin{construction}
  Let $\C$ and $\X$ be $\infty$-categories and let $q \colon \C \to \X$ be a Cartesian
  and coCartesian fibration. Let a commutative diagram $\sigma$ in $\X$ be
  given, depicted as
  \[
    \begin{tikzcd}
      \wt Y	\arrow[r, "h"] \arrow[d, "g_{Y}"] \arrow[dr, phantom, "\sigma"]
      &\wt X \arrow[d,"g_{X}"]
      \\
      Y	\arrow[r, "f"]
      &X
    \end{tikzcd}.
  \]
  As described in \Cref{sec:straightening}, any Cartesian fibration induces a
  functor $\X^{\op} \to \Cat_{\infty}$ given by $X \mapsto \C_{X}$ on objects
  and $f \mapsto f^{*}$ on morphisms. Thus, by $q$ we induce a commutative
  diagram $\sigma^{*}$ in $\Cat_{\infty}$ depicted as
  \[
    \begin{tikzcd}
      \C_{\wt Y}	  	\arrow[dr, phantom, "\sigma^{*}"]
      &\C_{\wt X} \arrow[l, "h^{*}"']
      \\\C_{Y}	\arrow[u, "g_{Y}^{*}"']
      &\C_{X} \arrow[l, "f^{*}"'] \arrow[u, "g_{X}^{*}"']
    \end{tikzcd}.
  \]
  As $q$ in addition to a Cartesian fibration is a coCartesian fibration, in
  particular the functors $h^{*}$ and $f^{*}$ admit left adjoints $h_{!}$ and
  $f_{!}$, respectively. Thus, we obtain a Beck-Chevalley transformation
  $BC[\sigma^{*}] \colon h_{!}g_{Y}^{*} \to g_{X}^{*} f_{!}$.
  Abusing notation, we will also refer to this transformation as the
  \emph{Beck-Chevalley transformation} associated to $\sigma$ and denote it by
  $BC[\sigma]$, and if $BC[\sigma]$ is an equivalence we will say that $\sigma$
  \emph{satisfies the Beck-Chevalley condition}.
\end{construction}

\begin{definition}[]
  Let $\C$ and $\X$ be $\infty$-categories. Given a map $q \colon \C \to \X$ of
  simplicial sets, we say that $q$ is a \emph{Beck-Chevalley fibration} if the
  following conditions are satisfied:
  \begin{enumerate}
    \item The map $q$ is a Cartesian and coCartesian fibration.

    \item $\X$ admits pullbacks.

    \item Every pullback square in $\X$ satisfies the Beck-Chevalley condition.
      That is, for every pullback square $\sigma$ in $\X$, depicted as
      \[
        \begin{tikzcd}
          \wt Y	\arrow[r, "h"] \arrow[d, "g_{Y}"] 	\arrow[dr, phantom, "\sigma"]
          &\wt X \arrow[d,"g_{Y}"]
          \\
          Y	\arrow[r, "f"]
          &X
        \end{tikzcd},
      \]
      the Beck-Chevalley transformation $BC[\sigma] \colon h_{!} g_{Y}^{*}
      \to g_{Y}^{*}f_{!}$ associated to $\sigma$ is an equivalence of functors
      from $\C_{Y}$ to $\C_{\wt X}$.
  \end{enumerate}
\end{definition}

\begin{example}[] \label{thm:local.systems.bc}
  Let $\C$ be an $\infty$-category which admits colimits. The Cartesian and
  coCartesian fibration $q \colon \LocSys(\C) \to \S$ of $\C$-valued local
  systems described in \Cref{thm:local.systems} is a Beck-Chevalley fibration
  -- see \Cref{thm:4.3.3.lurie}.
\end{example}

Given a Beck-Chevalley fibration $q \colon \C \to \X$ and a morphism $f \colon
Y \to X$ in $\X$, we are interested in the situation where $f_{!}$ in addition
to a left adjoint is also a right adjoint to $f^{*}$. We define a class of
morphisms in $\X$ for which this is the case, namely the \emph{ambidextrous}
morphisms of $\X$. The definition is done inductively on the truncation level of
the morphism.

\begin{construction} \label{def:n-ambidex}
  Let $q \colon \C \to \X$ be a Beck-Chevalley fibration. We define the
  following data for every $n\geq -2$:
  \begin{itemize}
    \item A collection of morphisms in $\X$ we call \emph{weakly
      $n$-ambidextrous morphisms}.

    \item For every weakly $n$-ambidextrous morphism $f \colon Y \to X$ in $\X$
      a natural transformation $\nu_{f}^{n} \colon f^{*}f_{!} \to
      \id_{\C_{Y}}$, which we refer to as the \emph{wrong way counit of $f$}
      (despite this transformation not being a counit of an adjunction in
      general).
    \item A collection of morphisms in $\X$ we call \emph{$n$-ambidextrous
      morphisms}, which includes all weakly $n$-ambidextrous morphisms.
    \item For every $n$-ambidextrous morphism $f \colon Y \to X$ in $\X$ the
      wrong way counit \emph{is} in fact a counit of an adjunction $f^{*}
      \dashv f_{!}$. Hence we induce a \emph{wrong way unit} $\mu_{f}^{n}
      \colon \id_{\C_{X}} \to f_{!}f^{*}$ which exhibits $f_{!}$ as a right
      adjoint to $f^{*}$.
  \end{itemize}
  These collections are defined inductively. Let $f \colon Y \to X$ be a
  morphism in $\X$. We say that $f$ is \emph{$(-2)$-ambidextrous} if $f$ is an
  equivalence (in this situation, we also say that $f$ is \emph{weakly
  $(-2)$-ambidextrous}). Thus, the induced adjoints $f_{!}$ and $f^{*}$ from
  $f$ are mutually inverse equivalences,  and the unit $\eta_{f}$ and counit
  $\varepsilon_{f}$ are equivalences. We define the \emph{wrong way counit}
  $\smash{\nu_{f}^{-2}} \colon f^{*}f_{!} \to \id_{\C_{Y}}$ of $f$ as an
  inverse to $\eta_{f}$ and we define the \emph{wrong way unit}
  $\smash{\mu^{-2}_{f}}$ of $f$ as an inverse to $\varepsilon_{f}$.

  Now, suppose that the notion of an $n$-ambidextrous morphism has been defined
  for some $n\geq -2$ and that for all $n$-ambidextrous morphisms a wrong way
  unit has been defined. Let again $f \colon Y \to X$ be a morphism in $\X$,
  and denote the pullback square of $f$ along itself by $\sigma$. Let $\Delta
  \colon Y \to Y \times_{X} Y$ be the diagonal map, defined as the map fitting
  into the commuting diagram
    \[
      \begin{tikzcd}
        Y \arrow[drr, bend left, "\id_{Y}"] \arrow[ddr, bend right, "\id_{Y}"] \arrow[dr, "\Delta"]
        & & \\
        & Y\times_{X} Y \arrow[r, "\pi_1"] \arrow[d, "\pi_2"] \arrow[dr,
        phantom, "\sigma"]
        & Y \arrow[d,
        "f"] \\
        & Y \arrow[r, "f"] & X
      \end{tikzcd}.
    \]
    We say that $f$ is \emph{weakly $(n+1)$-ambidextrous} if $\Delta$ is
    $n$-ambidextrous. As $q \colon \C \to \X$ was assumed to be a
    Beck-Chevalley fibration, $\sigma$ satisfies the Beck-Chevalley condition,
    so $BC[\sigma]$ is an equivalence. We can thus define a natural
    transformation $\nu_{f}^{n+1} \colon f^{*}f_{!} \to \id_{\C_{Y}}$ as the
    composite
    \[
      \nu_{f} ^{n+1} \colon  f^{*}f_{!} \xrightarrow{BC[\sigma]^{-1}}
      \pi_{1!}\pi_2^{*} \xrightarrow{\pi_{1!}\mu_{\Delta}^{n}\pi_2^{*}}
      \pi_{1!}\Delta_{!}\Delta^{*}\pi_2^{*} \simeq \id_{\C_{Y}},
    \]
    using that $\pi_{i}\Delta \simeq \id_{Y}$, and that applying both $(-)_{!}$
    and $(-)^{*}$ is functorial (see \Cref{sec:straightening}). We refer to
    $\nu_{f}^{n+1}$ as the \emph{wrong way counit map} of $f$ (note that this
    is abusive terminology as $\nu_{f}^{n+1}$ is only guaranteed to be a counit
    when $f$ is $(n+1)$-ambidextrous).

    We say that $f$ is \emph{$(n+1)$-ambidextrous} if the following
    condition holds: For every pullback square
    \[
      \begin{tikzcd}
        \wt Y	\arrow[r, "h"] \arrow[d, ""] 	&		\wt X \arrow[d,""] \\
        Y	\arrow[r, "f"]							&		X
      \end{tikzcd}
    \]
    in $\X$, the morphism $h$ is weakly $(n+1)$-ambidextrous and the natural
    transformation \smash{$\nu_{h}^{n+1}$} is the counit of an adjunction
    $h^{*} \dashv h_{!}$. We define the \emph{wrong way unit} $\mu_{h}^{n+1}
    \colon \id_{\C_{\wt Y}} \to h^{*}h_{!}$ of $h$ as a
    compatible unit of the adjunction determined by \smash{$\nu_{h}^{n+1}$}.

    If the Beck-Chevalley fibration $q \colon \C \to \X$ in question is
    ambiguous, we will instead say that $f$ is \emph{(weakly) $n$-ambidextrous
    with respect to $q$}, and we will refer to $\nu_{f}^{n}$ and $\mu_{f}^{n}$
    as \emph{the wrong way (co)unit of $f$ with respect to $q$.}
\end{construction}

\begin{remark} \label{thm:ambidex.properties}
  Let $q\colon \C \to \X$ be a Beck-Chevalley fibration and let $f\colon Y \to
  X$ be a morphism in $\X$. We make the following observations from
  \Cref{def:n-ambidex}:
  \begin{enumerate}
    \item If $f$ is $n$-ambidextrous, then $f$ is weakly $n$-ambidextrous. Indeed,
      for $n=-2$ this was by definition, and for $n\geq -1$, the square
      \[
        \begin{tikzcd}
          Y	\arrow[r, "f"] \arrow[d, "\id_{Y}"]
          &X \arrow[d,"\id_{X}"]
          \\
          Y	\arrow[r, "f"]
          &X
        \end{tikzcd}
      \]
      is a pullback square.

    \item The same pullback
      square as above also shows that if $f$ is $n$-ambidextrous, then the
      wrong way counit map $\nu_{f}^{n}$ determines an adjunction $f^{*}
      \dashv f_{!}$.

    \item If $f$ is $n$-ambidextrous, then any pullback of $f$ is
      $n$-ambidextrous. Indeed, if $h$ is a pullback of $f$, then any
      pullback of $h$ is also a pullback of $f$, by the pasting law for
      pullback squares (the dual of \cite[Lemma 4.4.2.1]{htt}). Thus, the
      conditions to be checked on pullbacks of $h$ follows from
      $n$-ambidexterity of $f$.

    \item If $f$ is weakly $n$-ambidextrous, then any pullback of $f$ is weakly
      $n$-ambidextrous. For $n=-2$ this is because pullbacks preserve
      equivalences. For $n \geq -1$, if $h \colon \wt Y \to \wt X$ is a
      pullback of $f$, denote the diagonal map of $f$ by $\Delta_{f} \colon Y
      \to Y \times_{X} Y$ and denote the diagonal map of $h$ by $\Delta_{h}
      \colon \wt Y \to \wt Y \times_{\wt X} \wt Y$.
      In this situation, we can form a pullback square of the form
      \[
        \begin{tikzcd}
          \wt Y	\arrow[r, "\Delta_{h}"] \arrow[d, ""]
          &\wt Y \times_{\wt X} \wt Y \arrow[d,""]
          \\
          Y	\arrow[r, "\Delta_{f}"]
          &Y \times_{X} Y
        \end{tikzcd}.
      \]
      As $f$ is weakly $n$-ambidextrous, $\Delta_{f}$ is $(n-1)$-ambidextrous so
      by the previous remark $\Delta_{h}$ is $(n-1)$-ambidextrous which by
      definition means that $h$ is weakly $n$-ambidextrous.
  \end{enumerate}
\end{remark}

Recall that a morphism $f \colon Y \to X$ in any $\infty$-category $\X$ is
called \emph{$n$-truncated} for $n\geq -2$ if for every $Z \in \X$, the map
$\Map_\X(Z,Y) \to \Map_{\X} (Z, X)$ induced by post-composition with $f$ has
$n$-truncated homotopy fibers (\cite[Definition 5.5.6.8]{htt}).

\begin{proposition}[] \label{thm:ambidex.truncated}
  Let $q \colon \C \to \X$ be a Beck-Chevalley fibration and let $f \colon Y
  \to X$ be a morphism in $\X$. If $f$ is (weakly) $n$-ambidextrous, then $f$
  is $n$-truncated.
\end{proposition}

\begin{proof}
  For $n=-2$ this is clear. For $n\geq -1$, this follows from \cite[Lemma
  5.5.6.15]{htt} which states that a morphism $f \colon Y \to X$ in $\X$ is
  $n$-truncated for $n\geq -1$ if and only if the diagonal $\Delta \colon Y \to
  Y \times_{X} Y$ is $(n-1)$-truncated.
\end{proof}

We will need to establish some coherence between (weak) ambidexterity of
different degrees.
\begin{proposition}\label{thm:ambidex.coherence}
  Let $q \colon \C \to \X$ be a Beck-Chevalley fibration, let $f \colon Y
  \to X$ be a morphism in $\X$, and let integers $-2 \leq m \leq n$ be given.
  \begin{enumerate}
    \item If $f$ is weakly $m$-ambidextrous, then $f$ is weakly $n$-ambidextrous
      and the wrong way counit maps $\nu_{f}^{m}$ and
      $\nu_{f}^n$ are
      equivalent.
    \item If $f$ is $m$-ambidextrous, then $f$ is $n$-ambidextrous and the
      wrong way unit maps $\mu_{f}^{m}$ and $\mu_{f}^{m}$ are equivalent.
    \item If $f$ is (weakly) $n$-ambidextrous, then $f$ is (weakly)
      $m$-ambidextrous if and only if $f$ is $m$-truncated.
  \end{enumerate}
\end{proposition}

\begin{proof}
  Omitted. See \cite[Proposition 4.1.10 (4)-(6)]{ambidexLurie}.
\end{proof}

\begin{definition}[] \label{def:ambidex}
  Let $q \colon \C \to \X$ be a Beck-Chevalley fibration and let $f$ be a
  morphism in $\X$.
  \begin{enumerate}
    \item We say that $f$ is \emph{weakly ambidextrous} if $f$ is weakly
      $n$-ambidextrous for some integer $n\geq -2$. In this situation, we denote
      the \emph{wrong way counit} $\nu_{f}^n$ of $f$ by $\nu_{f}$, which is
      well-defined up to equivalence by \Cref{thm:ambidex.coherence} (1).

    \item We say that $f$ is \emph{ambidextrous} if $f$ is $n$-ambidextrous for
      some integer $n\geq -2$. In this situation, we denote the \emph{wrong way
      unit} $\mu_{f}^n$ of $f$ by $\mu_{f}$, which is well-defined up to equivalence
      by \Cref{thm:ambidex.coherence} (2).
  \end{enumerate}
  If the Beck-Chevalley fibration $q \colon \C \to \X$ in question is
  ambiguous, we will instead say that $f$ is \emph{(weakly) ambidextrous with
  respect to $q$}, and refer to $\nu_{f}$ and $\mu_{f}$ as \emph{the wrong way
    (co)unit with respect to $q$}.
\end{definition}

\begin{definition}
  Let $q \colon \C \to \X$ be a Beck-Chevalley fibration and let $f \colon Y
  \to X$ be a morphism in $\X$. Assume that $f$ is weakly ambidexterous and
  that $f^{*}$ admits a right adjoint $f_{*}$. In this situation, we define the
  \emph{norm map} $\Nm_{f}$ associated to $f$ to be the composite
  \[
    \Nm_{f} \colon f_{!} \xrightarrow{\eta_{*f} f_{!}} f_{*}f^{*}f_{!}
    \xrightarrow{f_{*}\nu_{f}} f_{*}.
  \]
  That is, $\Nm_{f}$ is the mate of $\nu_{f} \colon f^{*}f_{!} \to
  \id_{\C_{Y}}$ under the adjunction $f^{*}\dashv f_{*}$.
\end{definition}

\begin{remark}[]
  Let $q \colon \C \to \X$ be a Beck-Chevalley fibration, let $f \colon Y \to
  X$ be a
  weakly ambidextrous morphism in $\X$, and assume that $f^{*}$ admits a right
  adjoint $f_{*}$. In this situation, $\Nm_{f}$ determines and is determined by
  $\nu_{f}$ under the adjunction $f^{*} \dashv f_{*}$. In particular, $\Nm_{f}$
  is equivalent to the identity if and only if $\nu_{f}$ is equivalent to the
  counit of this adjunction.
\end{remark}

\begin{remark}[]
  Let $q\colon \C \to \X$ be a Beck-Chevalley fibration and $f \colon Y \to X$
  a morphism in $\X$. If $f$ is ambidextrous, then the wrong way counit
  $\nu_{f} \colon f^{*}f_{!} \to \id_{\C_{Y}}$ exhibits $f_{!}$ as a right
  adjoint to $f^{*}$. In particular, the norm map $\Nm_{f}$ exists and is an
  equivalence.
\end{remark}

\begin{example}[]
  Let $\C$ be an $\infty$-category which admits colimits and limits. The
  Beck-Chevalley fibration $q \colon \LocSys(\C) \to \S$ of $\C$-valued local
  systems described in  \Cref{thm:local.systems} and
  \Cref{thm:local.systems.bc} has the property, that for any map $f\colon B \to
  A$ of spaces, the induced precomposition functor $f^{*} \colon \Fun(A, \C)
  \to \Fun(B, \C)$ admits both a left adjoint $f_{!}$ and a right adjoint
  $f_{*}$. Thus, for such local systems, the norm map $\Nm_{f} \colon f_{!} \to
  f_{*}$ always exists when $f$ is weakly ambidextrous.
\end{example}

%% file: normSquare/bcMaps.tex
In this section, we prove properties of Beck-Chevalley transformations
and norm maps, which will be needed in the proof of \Cref{thm:norm.square}.

\begin{remark}[] \label{thm:nat.trans}
  We first make a simple observation regarding natural transformations.
  Consider four functors of $\infty$-categories, $F, \smash{\wt F}\colon \C \to \D$ and $
  G, \smash{\wt G} \colon \D \to \E$ with natural transformations $\alpha, \beta$
  between them in the following way:
  \[
    \begin{tikzcd}
      \C \arrow[r, bend left, "F"] \arrow[r, bend right, "\wt F"'] \arrow[r,
      phantom, "\alpha \Downarrow"] & \D \arrow[r, bend left, "G"] \arrow[r,
      bend right, "\wt G"'] \arrow[r, phantom, "\beta \Downarrow"] & \mathcal{E}
    \end{tikzcd}.
  \]
  In this situation, there are two ways of composing $\alpha$ and $\beta$ into a
  transformation $GF \Rightarrow \wt G \wt F$ corresponding to the two ways
  around the diagram
  \[
    \begin{tikzcd}
      GF	\arrow[r, "G \alpha"] \arrow[d, "\beta F"] 	&		G \wt F \arrow[d,"\beta
      \wt F"] \\
      G \wt F	\arrow[r, "\wt G \alpha"]							&		\wt G \wt F
    \end{tikzcd}.
  \]
  As $\beta$ is a natural transformation this diagram commutes, so both
  composites can be identified. It follows that whenever we have a diagram
  where composites of transformations transform distinct functors, we can swap
  the order of these transformations.
\end{remark}

\begin{notation}[]
  In \Cref{thm:bc.pasting}, \Cref{thm:bc.pasting.v}, and \Cref{thm:bc.calculus}
  below, we will use the suggestive notation $f^{*}$ for a functor between
  $\infty$-categories, even when $f^{*}$ is not induced as the Cartesian lift
  of a morphism $f$.
  Furthermore, we will employ the same notation as in \Cref{thm:unit.counit}
  for (co)units of adjunctions, for example writing $\eta_{f}$ for the
  unit of an adjunction $f_{! }\dashv f^{*}$ rather than $\eta_{f^{*}}$
\end{notation}

We now show how Beck-Chevalley transformations behave under the pasting of
diagrams.
\begin{lemma}[] \label{thm:bc.pasting}
  Let a commuting diagram of functors between $\infty$-categories be given, of the form
   \begin{equation}\label{eq:bc.pasting}
     \begin{tikzcd}
      \mathcal{H} \arrow[dr, phantom, "\sigma"']
      & \mathcal{F}  \arrow[l, "k^{*}"'] \arrow[dr, phantom, "\tau"']
      & \D \arrow[l, "i^{*}"']
      \\
      \mathcal{G} \arrow[u, "l^{*}"']
      & \mathcal{E} \arrow[l, "j^{*}"'] \arrow[u, "h^{*}"']
      &  \C \arrow[l, "f^{*}"'] \arrow[u, "g^{*}"']
    \end{tikzcd}.
  \end{equation}
  \begin{enumerate}
    \item
  Suppose that all the horizontal functors admit left adjoints (denoted
  by $(-)_{!}$). In this situation, the composite transformation
  \[
    i_{!}k_{!}l^{*} \xrightarrow{i_{!} BC_{!}[\sigma]} i_{!}h^{*}j_{!}
    \xrightarrow{BC_{!}[\tau] j_{!}} g^{*}f_{!}j_{!}
  \]
  is equivalent to the Beck-Chevalley transformation of left adjoints associated to the outer
  square.
\item
  Suppose that all the horizontal functors admit right adjoints (denoted by
  $(-)_{*}$). In this situation, the composite transformation
  \[
    g^{*}f_{*}j_{*} \xrightarrow{BC_{*}[\tau]j_{*}} i_{*}h^{*}j_{*}
    \xrightarrow{i_{*}BC_{*}[\sigma]} i_{*}k_{*}l^{*}
  \]
  is equivalent to the Beck-Chevalley transformation of right adjoints
  associated to the outer square.
  \end{enumerate}
\end{lemma}
\begin{proof}
  We prove (1), the proof of (2) can be done by a similar argument. As the
  composition of adjoints produces adjoints, we see that the left adjoints of
  the functors $k^{*}i^{*}$ and $j^{*}f^{*}$ exist and are canonically
  equivalent to $i_{!}k_{!}$ and $f_{!}j_{!}$ respectively. Write $\eta_{fj}$
  for the unit of the adjunction  $f_{!}j_{!} \dashv j^{*}f^{*}$ and write
  $\varepsilon_{ik}$ for the counit of the adjunction $i_{!}k_{!} \dashv
  k^{*}i^{*}$. It follows that $\eta_{fj}$ can canonically be identified with
  the composite
  \[
    \id_{\mathcal{G}} \xrightarrow{\eta_{j}} j^{*}j_{!} \xrightarrow{j^{*}\eta_{f}j_{!}}
    j^{*}f^{*}f_{!}j_{!}
  \]
  and $\varepsilon_{ik}$ can canonically be identified with the composite
  \[
    i_!k_!k^{*}i^{*} \xrightarrow{i_{!}\varepsilon_{k}i^{*}} i_{!}i^{*}
    \xrightarrow{\varepsilon_{i}} \id_{\D}
  \]
  Consider the diagram
  \[
    \begin{tikzcd}[row sep={6em,between origins}, column sep={8em,between
      origins}]
      i_{!}k_{!}l^{*} \arrow[r, "\eta_{fj}"] \arrow[d, "\eta_{j}"]
      & i_{!}k_{!}l^{*}j^{*}f^{*}f_{!}j_{!} \arrow[rr, "\simeq"] \arrow[dr,
      "\simeq"]
      && i_{!}k_{!}k^{*}i^{*}g^{*}f_{!}j_{!} \arrow[r, "\varepsilon_{ik}"]
      \arrow[dr, "\varepsilon_{k}"description]
      & g^{*}f_{!}j_{!}
   \\ i_{!}k_{!}l^{*}j^{*}j_{!} \arrow[dr, "\simeq"] \arrow[ur,
   "\eta_{f}"description]
      && i_{!}k_{!}k^{*}h^{*}f^{*}f_{!}j_{!} \arrow[ur, "\simeq"] \arrow[dr,
      "\varepsilon_{k}"]
      &&i_{!}i^{*}g^{*}f_{!}j^{*} \arrow[u, "\varepsilon_{i}"]
   \\
      & i_{!}k_{!}k^{*}h^{*}j_{!} \arrow[r, "\varepsilon_{k}"] \arrow[ur,
      "\eta_{f}"]
      & i_{!}h^{*}j_{!} \arrow[r, "\eta_{f}"]
      & i_{!}h^{*}f^{*}f_{!}j^{*} \arrow[ur, "\simeq"]
      &
    \end{tikzcd}
  \]
  where the two outer composites are, respectively, the Beck-Chevalley
  transformation of the joined square of \cref{eq:bc.pasting} and the
  composite of $BC_{!}[\sigma]$ with $BC_{!}[\tau]$. The top right
  and top left triangles commute by the preceeding discussion and the top middle
  diagram commutes as the equivalence witnessing commutativity of the joined
  square of \cref{eq:bc.pasting} is canonically equivalent to the composite of
  the equivalences witnessing commutativity of $\sigma$ and $\tau$. The
  remaining three diagrams commute by \Cref{thm:nat.trans}. Thus, the whole
  diagram commutes, as desired.
\end{proof}
The analog of \Cref{thm:bc.pasting} for vertical pasting also holds. The
statement can be found as \cite[Lemma 2.1.6]{1811}.

\begin{lemma}[] \label{thm:bc.pasting.v}
  Let a commutative diagram of functors between $\infty$-categories be given,
  of the form
  \[
    \begin{tikzcd}[arrows=swap]
      \mathcal{H} \arrow[dr, phantom, "\sigma"]
      & \mathcal{G}  \arrow[l,
      "l^{*}"] \\ \mathcal{F} \arrow[dr, phantom, "\tau"]\arrow[u, "j^{*}"]  & \mathcal{E} \arrow[u,
      "k^{*}"] \arrow[l, "i^{*}"]  \\ \D \arrow[u, "h^{*}"] & \C \arrow[u,
      "g^{*}"] \arrow[l, "f^{*}"]
  \end{tikzcd}.
  \]
\begin{enumerate}
  \item Suppose that all the horizontal functors admit left adjoints (denoted by
    $(-)_{!}$). In this situation, the composite transformation
    \[
      l_{!}j^{*}h^{*} \xrightarrow{BC_{!}[\sigma]h^{*}} k^{*}i_{!}h^{*}
      \xrightarrow{k^{*}BC_{!}[\tau]} k^{*}g^{*}f_{!}
    \]
    is equivalent to the Beck-Chevalley transformation of left adjoints
    associated to the outer square.

  \item Suppose that all the horizontal functors admit right adjoints (denoted
    by $(-)_{*}$). In this situation, the composite transformation
    \[
      k^{*}g^{*}f_{*} \xrightarrow{k^{*}BC_{*}[\tau]} k^{*}i_{*}h^{*}
      \xrightarrow{BC_{*}[\sigma]h^{*}}l_{*}j^{*}h^{*}
    \]
    is equivalent to the Beck-Chevalley transformation of right adjoints
    associated to the outer square.
\end{enumerate}
\end{lemma}

\begin{proof}
  Omitted. The proof can done by a similar argument as the proof of
  \Cref{thm:bc.pasting}.
\end{proof}

The following commutativity properties between Beck-Chevalley transformations
and (co)units of adjunctions will be essential in the proof of
\Cref{thm:norm.square}.
\begin{lemma}[] \label{thm:bc.calculus}
  Let a commutative square
  \[
    \begin{tikzcd}[ampersand replacement = \&]
      \mathcal{F}	  	\&		\E \arrow[l, "i^{*}"']  \\
      \D	\arrow[u, "h^{*}"'] \&	\C \arrow[l, "f^{*}"'] \arrow[u, "g^{*}"']
    \end{tikzcd}
  \]
  of functors be given, and let $BC_{!}$ and $BC_{*}$ denote the two
  Beck-Chevalley transformations of this diagram (when they exist).
  \begin{enumerate}
    \item If $i^{*}$ and $f^{*}$ admit left adjoints $i_{!}$ and
      $f_{!}$, respectively, then $BC_{!}$ exists and the diagrams
      \begin{align*}
        \begin{tikzcd}[ampersand replacement=\&]
          i_{!} h^{*} f^{*} \arrow[r, "BC_{!}f^{*}"] \arrow[d, "\simeq"]
          \& g^*f_{!}f^* \arrow[d, "g^* \varepsilon_{f}"] \\
          i_{!}i^* g^* \arrow[r, "\varepsilon_{i}g^{*}"]
          \& g^*
        \end{tikzcd} \quad \textup{(1a)}
        && \text{and}
        &&
        \begin{tikzcd}[ampersand replacement=\&]
          h^{*}	\arrow[r, "\eta_{i}h^{*}"] \arrow[d, "h^* \eta_{f}"]
          \&		i^*i_{!}h^* \arrow[d,"i^* BC_{!}"] \\
          h^* f^*f_{!}	\arrow[r, "\simeq"]							\&		i^* g^*
          f_{!}
        \end{tikzcd} \quad \textup{(1b)}
      \end{align*}
      commute.

    \item If $i^*$ and $f^*$ admit right adjoints $i_{*}$ and
      $f_{*}$, respectively, then $BC_{*}$ exists and the diagrams
      \begin{align*}
        \begin{tikzcd}[ampersand replacement=\&]
          i^* g^* f_{*} \arrow[r, "i^* BC_{*}"] \arrow[d, "\simeq"]
          \& i^* i_{*} h^* \arrow[d, "\varepsilon_{*i} h^*"]
          \\
          h^* f^* f_{*} \arrow[r, "h^{*} \varepsilon_{* f}"]\& h^*
        \end{tikzcd} \quad \textup{(2a)}
        && \text{and}
        &&
        \begin{tikzcd}[ampersand replacement=\&]
          g^*	\arrow[r, "g^* \eta_{* f}"] \arrow[d, "\eta_{*i} g^*"]
          \&		g^* f_{*} f^* \arrow[d,"BC_{*} f^*"]
          \\ i_{*} i^* g^* \arrow[r, "\simeq"]
          \&i_{*} h^* f^*
        \end{tikzcd} \quad \textup{(2b)}
      \end{align*}
      commute.
  \end{enumerate}
\end{lemma}
\begin{proof}
  We show that (1a) commutes. Commutativity of the other diagrams can be shown
  by similar proofs. Writing out the definition of $BC_{!}$ and adding
  intermediate maps yields the diagram
  \[
\begin{tikzcd}
i_! h^* f^* \arrow[r, "i_{!}h^{*}\eta_f f^{*}"] \arrow[ddd, "\simeq"]
\arrow[rd, equals]
& i_! h^* f^* f_! f^* \arrow[r, "\simeq"] \arrow[d, "i_{!}h^{*}f^{*}\varepsilon_f"]
& i_! i^* g^* f_! f^* \arrow[r, "\varepsilon_{i}g^{*}f_{!}f^{*}"] & g^* f_! f^*
\arrow[ddd, "g^{*}\varepsilon_f"]
\\ & i_! h^* f^* \arrow[rd, "\simeq"] & &
\\ & & i_! i^* g^* \arrow[rd, "\varepsilon_{i}g^{*}"] &
\\ i_! i^* g^* \arrow[rrr, "\varepsilon_{i}g^{*}"] & & & g^*
\end{tikzcd}.
\]
The bottom left triangle commutes strictly, the top right diagram commutes by
\Cref{thm:nat.trans}, as $\varepsilon_{f}$ transform distinct functors from
$\simeq$ and $\varepsilon_{i}$, and the top left triangle commutes by the
triangle identity for adjoints. Thus, the whole diagram commutes.
\end{proof}

We will also need the following result in the proof of \Cref{thm:norm.square},
relating the norm map of an ambidextrous morphism to the wrong way unit:
\begin{lemma}\label{thm:nm.calculus}
  Let $q \colon \C \to \X$ be a Beck-Chevalley fibration and let $f \colon Y
  \to X$ be a morphism in $\X$. If $f$ is ambidextrous, then the diagram
  \[
      \begin{tikzcd}
        & f_{!}f^{*} \arrow[dd, "\Nm_{f}f^{*}"]
        \\
        \id_{\C_{X}} \arrow[ur, "\mu_{f}"] \arrow[dr, "\eta_{*f}"']
        &
        \\
        &
        f_{*}f^{*}
      \end{tikzcd}
  \]
  commutes.
\end{lemma}
\begin{proof}
    Existence of the wrong-way unit $\mu_{f}$ and right adjoint $f_{*}$
    follows from $f$ being ambidextrous. Writing out the definition of
    $\Nm_{f}$ and adding intermediate maps yields the diagram
  \[
    \begin{tikzcd}[row sep=huge]
      &
      &
      f_{!}f^{*} \arrow[d, "\eta_{*f}f_{!}f^{*}"]
      \\
      \id_{\C_{X}} \arrow[urr, bend left, "\mu_{f}"] \arrow[r, "\eta_{*f}"]
      \arrow[drr, bend right, "\eta_{*f}"]
      & f_{*}f^{*} \arrow[r, "f_{*}f^{*}\mu_{f}"] \arrow[dr, equals, ""]
      & f_{*}f^{*}f_{!}f^{*} \arrow[d, "f_{*}\nu_{f}f^{*}"]
      \\
      && f_{*}f^{*}
    \end{tikzcd}.
  \]
  The upper diagram commutes by naturality of $\eta_{*f}$, the lower left
  triangle commutes strictly, and the lower right triangle commutes by the
  triangle identity for adjoints, as the ambidexterity of $f$ implies that
  $\mu_{f}$ and $\nu_{f}$ are unit and counit, respectively, of an adjunction
  $f^{*} \dashv f_{!}$.
\end{proof}

%% file: normSquare/normMap.tex
We now turn to the proof of our main theorem, namely that given two
Beck-Chevalley fibrations $\C \to \X$ and $\D \to \X$ and a functor $F \colon
\C \to \D$ which preserves cartesian edges, the resulting \emph{norm square}
commutes. Our proof is inspired by the proof of \cite[Proposition 4.1.1
(1)]{ambidexLurie}.

\begin{theorem}[]\label{thm:norm.square}
   Let $q \colon \C \to \X$ and $p \colon \D \to \X$ be Beck-Chevalley
   fibrations and let $F \colon \C \to \D$ be a functor of $\infty$-categories
   which sends $q$-Cartesian edges in $\C$ to $p$-Cartesian edges in $\D$.
   Suppose the diagram
\[
  \begin{tikzcd}[column sep=normal]
    \C \arrow[rr, "F"] \arrow[dr, "q"'] && \D \arrow[dl, "p"] \\
                      & \X &
  \end{tikzcd}
\]
commutes. Let $f \colon Y \to X$ be a morphism in $\X$, denote the Cartesian
lifts of $f$ by $f^{*} \colon \C_{X} \to \C_{Y}$ and $\d{f^{*}} \colon \D_{X}
\to \D_{Y}$ and denote the coCartesian lifts of $f$ by $f_{!} \colon \C_{Y} \to
\C_{X}$ and $\d{f_{!}} \colon \D_{Y} \to \D_{X}$. Furthermore, letting $\Delta
\colon Y \to Y \times_{X} Y$ denote the diagonal map induced by $f$, denote the
Cartesian lifts of $\Delta$ by $\Delta^{*} \colon \C_{Y \times_{X} Y} \to
\C_{Y}$ and $\d{\Delta^{*}} \colon \D_{Y \times_{X} Y} \to \D_{Y}$. Assume the
following hold:

\begin{itemize}
    \item $f^{*}$ and $\d {f^{*}}$ admit right adjoints $f_{*}$ and $\d{f_{*}}$
        respectively.
    \item $f$ is weakly ambidextrous with respect to $q$ and $p$.
    \item
        The square
        \[
        \begin{tikzcd}
            \D_{Y}
            &\D_{Y\times_{X}Y} \arrow[l, "\d{\Delta^{*}}"']
            \\\C_{Y}	\arrow[u, "F"']
            &\C_{Y\times_{X}Y} \arrow[l, "\Delta^{*}"'] \arrow[u, "F"']
        \end{tikzcd}
    \]
        satisfies the Beck-Chevalley condition.
\end{itemize}
In this situation, the diagram
\begin{equation}\label{eq:norm.square}
  \begin{tikzcd}
    \d{f_{!}}F  \arrow[d, "\Nm_{\d{f}}F"] \arrow[r, "BC_{!}"]
    & Ff_{!} \arrow[d, "F\Nm_{f}"] \\
    \d{f_{*}}F &  Ff_{*} \arrow[l, "BC_{*}"']
  \end{tikzcd}
\end{equation}
of natural transformations of functors in $\Fun(\C_{Y}, \D_{X})$ commutes,
where $BC_{!}$ and $BC_{*}$ are the two Beck-Chevalley transformations of the
diagram
\[
  \begin{tikzcd}[ampersand replacement = \&]
    \D_{Y}	\arrow[from=r, "\d{f^{*}}"']  	\&		\D_{X}  \\ \C_{Y}
    \arrow[from=r, "f^{*}"'] \arrow[u, "F"']						\& \C_{X}
    \arrow[u,"F"']
  \end{tikzcd}
\]
induced by $F$.
\end{theorem}

\begin{proof}
  Denote the diagram
\[
  \begin{tikzcd}[ampersand replacement = \&]
    \D_{Y}	\arrow[from=r, "\d{f^{*}}"'] \arrow[dr, phantom, "\tau"]  	\&		\D_{X}  \\
  \C_{Y}	\arrow[from=r, "f^{*}"'] \arrow[u, "F"']						\&
  \C_{X} \arrow[u,"F"']
  \end{tikzcd}
\]
by $\tau$. As described in \Cref{sec:straightening}, the functor $F$ induces a
natural transformation $F \colon \C_{(-)} \Longrightarrow \D_{(-)}$, so $\tau$
commutes. Hence, the Beck-Chevalley transformations $BC_{!}=BC_{!}[\tau]$ and
$BC_{*} = BC_{*}[\tau]$ exist.
For the following, we will indicate a natural transformations in $\D$ induced
by a morphism $g$ in $\X$ by writing $\d{g}$ in the subscript. For example, the
wrong way counit of $f$ in $\D$ is denoted by $\nu_{\d{f}}$. By the adjunction
$\d{f^{*}} \dashv \d{f_{*}}$, commutativity of \cref{eq:norm.square} is
equivalent to commutativity of
  \[
    \begin{tikzcd}
      \d{f^{*}} \d{f_{!}} F \arrow[rr, "\d{f^{*}}BC_{!}{[\tau]}"] \arrow[d,
      "\d{f^{*}}\Nm_{\d{f}}F"] & & \d{f^{*}} F f_{!} \arrow[d,
      "\d{f^{*}}F\Nm_{f}"] \\
      \d{f^*} \d{f_{*}} F \arrow[d, "\varepsilon_{* \d{f}}F"] & \d{f^*}
      \d{f_{*}} F \arrow[dl, "\varepsilon_{* \d{f}}F"] & \d{f^{*}}F f_{*}
      \arrow[l, "\d{f^{*}} BC_{*}{[\tau]}"'] \\
      F &&
    \end{tikzcd}.
  \]
  By adding some extra natural transformations, we can consider the
  diagram
  \[
    \begin{tikzcd}
      \d{f^{*}} \d{f_{!}} F \arrow[rr, "\d{f^{*}}BC_{!}{[\tau]}"] \arrow[d,
      "\d{f^{*}}\Nm_{\d{f}}F"] \arrow[dd, bend right=50, "\nu_{\d{f}}F"'] &
      & \d{f^{*}} F f_{!} \arrow[d, "\d{f^{*}}F\Nm_{f}"] \arrow[ddr, bend left,
      "\simeq"]
      &
      \\
      \d{f^*} \d{f_{*}} F \arrow[d, "\varepsilon_{* \d{f}}F"]
      & \d{f^*} \d{f_{*}} F \arrow[dl, "\varepsilon_{* \d{f}}F"]
      & \d{f^{*}}F f_{*} \arrow[l, "\d{f^{*}} BC_{*}{[\tau]}"'] \arrow[d,
      "\simeq"]
      &
      \\
      F
      &
      & F f^{*} f_{*} \arrow[ll, "F\varepsilon_{* f}"']
      & F f^{*} f_{!} \arrow[l, "Ff^{*}\Nm_{f}"'] \arrow[lll, bend left,
      "F\nu_{f}"']
    \end{tikzcd}.
  \]
  All the new subdiagrams commute automatically: The diagram on the left
  commutes as $\Nm_{\d{f}}$ is the mate of $\nu_{\d{f}}$, the diagram on the
  bottom commutes as $\Nm_{f}$ is the mate of $\nu_{f}$, the diagram on the
  right commutes by naturality of the equivalence $\simeq$, and the middle
  diagram commutes by \Cref{thm:bc.calculus} (2a). Thus, commutativity of
  \cref{eq:norm.square} is equivalent to commutativity of the diagram
  \begin{equation}\label{eq:norm.square.mate}
    \begin{tikzcd}
    \d{f^{*}} \d{f_{!}} F \arrow[r, "\d{f^{*}}BC{[\tau]}"] \arrow[d,
    "\nu_{\d{f}}F"] &  \d{f^{*}} F f_{!} \arrow[d, "\simeq"] \\ F & Ff^{*}f_{!}
    \arrow[l, "F \nu_{f}"']
    \end{tikzcd}.
  \end{equation}
  As $f$ is weakly ambidextrous, $f$ is weakly $n$-ambidextrous for some
  integer $n \geq -2$. We prove that \cref{eq:norm.square.mate} commutes by
  induction on $n$.

  First assume that $f$ is weakly $(-2)$-ambidextrous. In
  this situation, $f$ is an equivalence and $\nu_{f}$ is defined as an inverse
  of $\eta_{f}$. Showing commutativity of \cref{eq:norm.square.mate} is
  thus equivalent to showing that the diagram
  \[
    \begin{tikzcd}
    \d{f^{*}} \d{f_{!}} F \arrow[r, "\d{f^{*}} BC{[\tau]}"]
      & \d{f^{*}} F f_{!} \arrow[d, "\simeq"]
      \\ F \arrow[u, "\eta_{\d{f}}F"'] \arrow[r, "F\eta_{f}"]
      & F f^{*}f_{!}
    \end{tikzcd}
  \]
  commutes. It does so, by \Cref{thm:bc.calculus} (1b).

  Now, let $n\geq-1$, assume that $f$ is weakly $n$-ambidextrous, and assume that
  we have shown commutativity of \cref{eq:norm.square} for any weakly
  $(n-1)$-ambidextrous map. Consider the pullback square
  \[
    \begin{tikzcd}
      Y \times_{X} Y	\arrow[r, "\pi_1"] \arrow[d, "\pi_2"]
      &Y \arrow[d,"f"]
      \\
      Y	\arrow[r, "f"]
      &X
    \end{tikzcd}
  \]
  in $\X$.
  Denote the lifts of this square into $\C$ and $\D$ by $\sigma$ and
  $\d{\sigma}$ respectively. That is, let $\sigma$ and $\d{\sigma}$ denote the
  squares
  \begin{align*}
    \begin{tikzcd}[ampersand replacement = \&]
      \C_{Y\times_{X} Y}\arrow[dr, phantom, "\sigma"]	  	\&
      \C_{Y} \arrow[l, "\pi_1^{*}"']  \\
      \C_{Y} 	\arrow[u, "\pi_2^{*}"'] \&	\C_{X} \arrow[l, "f^{*}"'] \arrow[u,
      "f^{*}"']
    \end{tikzcd}
    && \text{and}
    && \begin{tikzcd}[ampersand replacement = \&]
      \D_{Y\times_{X} Y}\arrow[dr, phantom, "\d{\sigma}"]	  	\&		\D_{Y}
      \arrow[l, "\d{\pi_1^{*}}"']  \\
      \D_{Y} 	\arrow[u, "\d{\pi_2^{*}}"'] \&	\D_{X} \arrow[l, "\d{f^{*}}"']
      \arrow[u, "\d{f^{*}}"']
    \end{tikzcd}.
  \end{align*}
  As these squares are induced from a pullback square by Beck-Chevalley
  fibrations, the Beck-Chevalley transformations $BC[\sigma]$ and
  $BC[\d{\sigma}]$ are both equivalences.
  As $F$ preserves Cartesian edges, we from the diagonal map $\Delta$ and
  $\pi_1$ induce commuting squares $\rho$ and $\upsilon$ respectively, depicted
  as
  \begin{align*}
    \begin{tikzcd}[ampersand replacement = \&]
      \D_{Y} \arrow[dr, phantom,
      "\rho"]	  	\&		\D_{Y \times_{X}Y} \arrow[l, "\d{\Delta^{*}}"']  \\
      \C_{Y}	\arrow[u, "F"'] \&	\C_{Y \times_{X}Y} \arrow[l,
      "\Delta^{*}"'] \arrow[u, "F"']
    \end{tikzcd}
    &&
    \text{and}
    &&
    \begin{tikzcd}[ampersand replacement = \&]
      \D_{Y \times_{X} Y} \arrow[dr, phantom, "\upsilon"]	  	\&
      \D_{Y}
      \arrow[l, "\d{\pi_1^{*}}"']  \\
      \C_{Y \times_{X} Y}	\arrow[u, "F"'] \&	\C_Y \arrow[l,
      "\pi_1^{*}"'] \arrow[u, "F"']
    \end{tikzcd}.
  \end{align*}
  Also, by definition, as $f$ is weakly $n$-ambidextrous, $\Delta$ is
  $(n-1)$-ambidextrous. In particular, we have wrong way unit maps
  $\mu_{\Delta}$ and $\mu_{\d{\Delta}}$. Recall that the wrong way counit
  $\nu_{f}$ of $f$ is given as the composite
  \[
    \nu_{f} \colon f^{*}f_{!} \xrightarrow{BC[\sigma]^{-1}} \pi_{1!} \pi_{2}^{*}
    \xrightarrow{\pi_{1!} \mu_{\Delta}  \pi_{2}^{*}} \pi_{1!}
    \Delta_{!}\Delta^{*} \pi_2^{*} \simeq \id_{\C_{Y}},
  \]
  and likewise for $\nu_{\d{f}}$.
  Writing out the definition of $\nu_{f}$ and $\nu_{\d{f}}$ we can then show
  that \cref{eq:norm.square.mate} commutes by showing that the following
  diagram commutes:
  \[
    \begin{tikzcd}[row sep=huge, column sep=8em]
      \d{f^{*}} \d{f_{!}} F \arrow[ddddr, phantom, "(1)"] \arrow[dd, "
      \d{f^{*}}BC{[\tau]}"] \arrow[r, "{BC[\d{\sigma}]^{-1}}F"]
      & \d{\pi_{1!}} \d{\pi_{2}^{*}}F \arrow[dd, "\simeq"] \arrow[r,
      "\d{\pi_{1!}}\mu_{\d{\Delta}}  \d{\pi_{2}^{*}} F"{name=UU}]
      & \d{\pi_{1!}} \d{\Delta_{!}} \d{\Delta^{*}} \d{\pi_{2}^{*}}F
      \arrow[d, "\simeq"] \arrow[rdd, bend left, "\simeq"name=UR]
      &
      \\
      &
      & \d{\pi_{1!}} \d{\Delta_{!}}
      \d{\Delta^{*}} F \pi_{2}^{*} \arrow[d, "\simeq"]
      &
      \\ \d{f^{*}}
      F f_{!} \arrow[dd, "\simeq"]
      & \d{\pi_{1!}} F \pi_{2}^{*} \arrow[dd,
      "BC{[\upsilon]}\pi_{2}^{*}"] \arrow[ur, "\d{\pi_{1!}}
      \mu_{\d{\Delta}}F \pi_{2}^{*}"{name=MU, description}] \arrow[dr,
      "\d{\pi_{1!}} F\mu_{\Delta} \pi_{2}^{*}"{name=MD, description}]
      & \d{\pi_{1!}} \d{\Delta_{!}} F \Delta^{*} \pi_{2}^{*}  \arrow[d,
      "\d{\pi_{1!}} BC{[\rho]}\Delta^{*}\pi_{2}^{*}"] \arrow[r,
      "\simeq"]
      & F
      \\
      &
      & \d{\pi_{1!}}F \Delta_{!}\Delta^{*}\pi_{2}^{*}
      \arrow[d, "BC{[\upsilon]}\Delta_{!}\Delta^{*}\pi_{2}^{*}"]
      &
      \\ F
      f^{*}f_{!} \arrow[r, "FBC{[\sigma]}^{-1}"]
      & F\pi_{1!}\pi_{2}^{*}
      \arrow[r, "F\pi_{1!}\mu_{\Delta}\pi_{2}^{*}"{name=DD}]
      & F \pi_{1!}
      \Delta_{!}\Delta^{*}\pi_{2}^{*} \arrow[ruu, bend right, "\simeq"'name=DR]
      &
      \arrow[from=UU, to=MU, phantom, "(2)"]
      \arrow[from=MU, to=MD, phantom, "(3)"]
      \arrow[from=MD, to=DD, phantom, "(4)"]
      \arrow[from=UR, to=3-3, phantom, shift right=0.5em, "(5)"]
      \arrow[from=DR, to=3-3, phantom, shift left=0.5em, "(6)"]
    \end{tikzcd}
  \]
   We prove the diagram commutes by proving that each of the subdiagrams
   (1)-(6) commute:

  (2): Commutativity of diagram (2) follows from commutativity of
  the diagram
  \[
    \begin{tikzcd}[sep=large]
      \d{\pi_{2}^{*}} F \arrow[r, "\mu_{\d{\Delta}} \d{\pi_{2}^{*}} F"]
      \arrow[d, "\simeq"]  &\d{\Delta_!} \d{\Delta^{*}} \d{\pi_{2}^{*}} F
      \arrow[d, "\simeq"]\\ F\pi_{2}^{*} \arrow[r,
      "\mu_{\d{\Delta}}F\pi_{2}^{*}"] &\d{\Delta_{!}} \d{\Delta^{*}}
      F \pi_{2}^{*}
    \end{tikzcd},
  \]
  which commutes by naturality of $\mu_{\d{\Delta}}$.

  (4): Commutativity of diagram (4) follows from commutativity of
  the diagram
  \[
    \begin{tikzcd}[sep=large]
      \d{\pi_{1!}} F \arrow[r, "\d{\pi_{1!}}F\mu_{\Delta}"]
      \arrow[d, "BC{[\upsilon]}"] & \d{\pi_{1!}} F \Delta_{!}\Delta^{*}
      \arrow[d, "BC{[\upsilon]}\Delta_{!}\Delta^{*}"] \\
      F\pi_{1!} \arrow[r, "F\pi_{1!}\mu_{\Delta}"] &
      F\pi_{1!}\Delta_{!}\Delta^{*}
    \end{tikzcd},
  \]
  which commutes by naturality of $BC[\upsilon]$.

  (5):
  We can write out diagram (5) as
  \[
    \begin{tikzcd}
      \d{\pi_{1!}}\d{\Delta_{!}}\d{\Delta^{*}}\d{\pi_{2}^{*}}F \arrow[rdd, bend
      left, "\simeq"]
      \arrow[d, "\simeq"]
      &
      &
      \\ \d{\pi_{1!}}\d{\Delta_{!}}\d{\Delta^*}F \pi_{2}^{*} \arrow[d,
      "\simeq"]
      & & \\
      \d{\pi_{1!}}\d{\Delta_{!}} F\Delta^{*}\pi_{2}^{*} \arrow[r, "\simeq"]
      & \d{\pi_{1!}}\d{\Delta_{!}}F \arrow[r, "\simeq"]
      & F
    \end{tikzcd},
  \]
  and hence it suffices to show commutativity of the diagram
  \begin{equation}\label{eq:dia.5}
    \begin{tikzcd}
      \d{\Delta^{*}}\d{\pi_{2}^{*}}F \arrow[rdd, bend left, "\simeq"] \arrow[d,
      "\simeq"]
      &
      \\ \d{\Delta^{*}}F \pi_{2}^{*} \arrow[d, "\simeq"] &  \\
      F\Delta^{*}\pi_{2}^{*} \arrow[r, "\simeq"] & F
    \end{tikzcd}.
  \end{equation}
The equivalence $\pi_2 \Delta \simeq \id_{Y}$ in $\X$ can be expressed as a
2-simplex $\Delta^{2} \to \X$, depicted as
\[
  \begin{tikzcd}[column sep=normal]
    Y \arrow[rr, equals] \arrow[dr, "\Delta"]
    && Y
    \\
    & Y \times_{X} Y \arrow[ur, "\pi_2"]&
  \end{tikzcd}.
\]
As $F$ preserves Cartesian edges in $\C$, we induce a commuting prism
  \[
\begin{tikzcd}
\D_Y \arrow[rr, equals] \arrow[from=rd, "\d{\Delta^{*}}"']               & &
\D_Y \\
& \D_{Y \times_X Y} \arrow[from=ru, "\d{\pi_2^{*}}"']
& \\
\C_Y \arrow[rr, equals] \arrow[uu, "F"'] \arrow[from=rd, "\Delta^{*}"'] & & \C_Y
\arrow[uu, "F"'] \\
& \C_{Y \times_X Y} \arrow[from=ru, "\pi_2^{*}"'] \arrow[uu, crossing over,
"F"'{yshift=10pt}] &
\end{tikzcd},
\]
i.e. a map $\Delta^{2} \times \Delta^{1} \to \Cat_{\infty}$, which restricts to
the written verticies and edges. Each vertex of \cref{eq:dia.5} corresponds to
a path from $\C_{Y}$ to $\D_{Y}$ in this prism, and the edges of
\cref{eq:dia.5} are the equivalences witnessed by the prism. Thus,
\cref{eq:dia.5} commutes, which shows that diagram (5) commutes.

  (6): For the left vertical composite in this diagram, by
  \Cref{thm:bc.pasting} (1) the composite transformation
  \[
    \d{\pi_{1!}} \d{\Delta_{!}} F \xrightarrow{\d{\pi_{1!}} BC[\rho]}
    \d{\pi_{1!}} F \Delta_{!} \xrightarrow{BC[\upsilon] \Delta_{!}} F
    \pi_{1!} \Delta_{!}
  \]
  can be found as the Beck-Chevalley transformation of the joined square
  \[
    \begin{tikzcd}
      \D_{Y} \arrow[from=r, "\d{\Delta^{*}}"'] \arrow[dr, phantom, "\rho"] &
      \D_{Y \times_{X} Y } \arrow[from=r, "\d{\pi_{1}^{*}}"'] \arrow[dr,
      phantom, "\upsilon"] & \D_{X} \\
      \C_{Y} \arrow[u, "F"'] \arrow[from=r, "\Delta^{*}"'] & \C_{Y \times_{X} Y}
      \arrow[u, "F"'] \arrow[from=r, "\pi_{1}^{*}"'] &  \C_{Y} \arrow[u, "F"']
    \end{tikzcd}.
  \]
  As $\pi_{1} \Delta \simeq \id_{X}$ in $\X$ it follows that both the top
  and bottom horizontal compositions are equivalent to the identity. Thus, the
  Beck-Chevalley transformation of the joined square is equivalent to the
  identity transformation $F \simeq F$. This shows the commutativity of
  (6).

  (1):
  For (1) it suffices to show that
  \[
    \d{\pi_{1!}}\d{\pi_{2}^{*}}F \xrightarrow{BC[\d{\sigma}]F}\d{f^{*}}
    \d{f_{!}} F \xrightarrow{\d{f^{*}}BC[\tau]} \d{f^{*}} Ff_{!} \simeq
    Ff^{*}f_{!}
  \]
  and
  \[
    \d{\pi_{1!}}\d{\pi_{2}^{*}}F \simeq \d{\pi_{1!}} F \pi_2^{*}
    \xrightarrow{BC[\upsilon]\pi_{2}^{*}} F\pi_{1!}\pi_2^{*} \xrightarrow{F BC[\sigma]}
    Ff^{*}f_{!}
  \]
  are equivalent. As $F$ preserves Cartesian edges, we have a commutative cube
  with faces $\d{\sigma}$ and  $\sigma$ on top and bottom respectively, and
  $\nu$ and $\tau$ twice each on the sides, depicted as follows:
  \[
    \begin{tikzcd}[labels=description]
      \D_{X \times_Y X}
      &
      &
      & \D_X \arrow[lll, "\d{\pi_1^{*}}"]
      &
      \\
      & \D_X \arrow[lu, "\d{\pi_2^{*}}"]       &
      &
      & \D_Y  \arrow[lu, "\d{f^{*}}"]
      \\ \C_{X\times_Y X} \arrow[uu, "F"]
      &
      &
      &
      \C_X \arrow[lll, "\pi_1^{*}"]
      &
      \\
      & \C_X \arrow[lu, "\pi_2^{*}"]
      &
      &
      & \C_Y \arrow[lll, "f^{*}"] \arrow[uu, "F"] \arrow[lu, "f^{*}"]
      \arrow[from=3-4, to=1-4, "F" yshift=-1em]
      \arrow[from=4-5, to=1-4, dashed]
      \arrow[from=2-5, to=2-2, crossing over, "\d{f^{*}}"]
      \arrow[from=4-2, to=2-2, crossing over, "F" yshift=1em]
      \arrow[from=4-2, to=1-1, crossing over, dashed]
\end{tikzcd}
\]
As the cube commutes, that is, it represents a diagram $\Delta^{1} \times
\Delta^{1} \times \Delta^{1} \to \Cat_{\infty}$, we obtain a diagonal diagram,
indicated by the dashed lines. Denote this diagram by $\vartheta$. Pasting the
front face and top face, or pasting the bottom face and the back face, both
produce a square equivalent to $\vartheta$. By \Cref{thm:bc.pasting.v} (1) we
therefore see that both
\[
    \d{\pi_{1!}}\d{\pi_{2}^{*}}F \xrightarrow{BC[\d{\sigma}]F}\d{f^{*}}
    \d{f_{!}} F \xrightarrow{\d{f^{*}}BC[\tau]} \d{f^{*}} Ff_{!} \simeq
    Ff^{*}f_{!}
  \]
  and
\[
    \d{\pi_{1!}}\d{\pi_{2}^{*}}F \simeq \d{\pi_{1!}} F \pi_2^{*}
    \xrightarrow{BC[\upsilon]\pi_2^{*}} F\pi_{1!}\pi_2^{*} \xrightarrow{F BC[\sigma]}
    Ff^{*}f_{!},
  \]
  are equivalent to $BC[\vartheta]$, showing that (1) commutes.

(3): To show commutativity of (3) it suffices to show
  commutativity of
\[
  \begin{tikzcd}
    & \d{\Delta_{!}}\d{\Delta^{*}}F \arrow[d, "\simeq"] \\
    F \arrow[ur, "\mu_{\d{\Delta}}F"]
    \arrow[dr, "F\mu_{\Delta}"']
    & \d{\Delta_{!}} F \Delta^{*} \arrow[d, "BC{[\rho]}\Delta^{*}"] \\
    & F\Delta_{!}\Delta^{*}
  \end{tikzcd}.
\]
As $f$ is weakly ambidextrous it follows that the diagonal map $\Delta$ is
ambidextrous so the norm maps $\Nm_{\Delta}$ and $\Nm_{\d{\Delta}}$ are
equivalences. We can thus add intermediate transformations to obtain the
diagram
\[
\begin{tikzcd}[]
    & & & \d{\Delta_{!}}\d{\Delta^{*}}F \arrow[dd, "\simeq"] \\
    & & \d{\Delta_{*}}\d{\Delta^{*}}F \arrow[ru,
    "\Nm_{\d{\Delta}}^{-1}\d{\Delta^{*}}F"description] \arrow[d,
    "\simeq"description] & \\
    F \arrow[rrruu, bend left, "\mu_{\d{\Delta}}F"] \arrow[rru,
    "\eta_{*\d{\Delta}}F"description] \arrow[rrrdd, bend right, "F
    \mu_{\Delta}"'] \arrow[rrd, "F \eta_{*\Delta}"description] &  &
    \d{\Delta_{*}}F \Delta^{*} \arrow[r, "\Nm_{\d{\Delta}}^{-1}F\Delta^{*}"] &
    \d{\Delta_{!}}F \Delta^{*} \arrow[dd, "BC_{!}{[\rho]}\Delta^{*}"] \\
    & & F\Delta_{*}\Delta^{*} \arrow[rd,
    "F\Nm_{\Delta}^{-1}\Delta^{*}"description] \arrow[u,
    "BC_{*}{[\rho]}\Delta^{*}"description] & \\
    & & & F \Delta_{!}\Delta^{*}
\end{tikzcd}.
\]
All the subdiagrams commute: The upper left and lower left triangles both
commute by \Cref{thm:nm.calculus}, the middle triangle commutes by
\Cref{thm:bc.calculus} (2b), and the upper right diagram commutes by naturality
of $\Nm_{\d{\Delta}}^{-1}$. For the lower the lower right diagram, by the inductive
hypothesis the diagram
\[
    \begin{tikzcd}
        \d{\Delta_{!}}F	\arrow[r, "BC_{!}{[\rho]}"] \arrow[d,
        "\Nm_{\d{\Delta}}F"] & F\Delta_{!} \arrow[d,"F \Nm_{\Delta}"] \\
        \d{\Delta_{*}}F	&F\Delta_{*}  \arrow[l, "BC_{*}{[\rho]}"']
    \end{tikzcd}
\]
commutes. As we assumed $BC_{!}[\rho]$ to be an equivalence, and since $\Delta$
being ambidextrous implies that the two norm maps are equivalences, it follows
that $BC_{*}[\rho]$ must also be an equivalence. This implies that
$F\Nm_{\Delta}$ is inverse to to composite
$(BC_{!}[\rho])(\Nm_{\d{\Delta}}^{-1} F)(BC_{*}[\rho])$ and hence that the
lower right diagram commutes. Thus, diagram (3) commutes.

This proves the commutativity of \cref{eq:norm.square.mate}, and hence of
\cref{eq:norm.square}.
\end{proof}

 In most cases we are interested in, where this result would be used, $f^{*}$
 does admit a right adjoint $f_{*}$. There is, however, an analog result in the
 case where $f^{*}$ does not admit a right adjoint.
\begin{corollary} \label{thm:nu.square}
 Let $q \colon \C \to \X$ and $p \colon \D \to \X$ be Beck-Chevalley fibrations
 and let $F \colon \C \to \D$ be a functor of $\infty$-categories which sends
 $q$-Cartesian edges in $\C$ to $p$-Cartesian edges in $\D$. Suppose the
 diagram
 \[
   \begin{tikzcd}[column sep=normal]
   \C \arrow[rr, "F"] \arrow[dr, "q"'] && \D  \arrow[dl, "p"]
   \\ & \X &
   \end{tikzcd}
 \]
commutes. Let $f \colon Y \to X$ be a morphism in $\X$, denote the Cartesian
lifts of $f$ by $f^{*} \colon \C_{X} \to \C_{Y}$ and $\d{f^{*}} \colon \D_{X}
\to \D_{Y}$ and denote the coCartesian lifts of $f$ by $f_{!} \colon \C_{Y} \to
\C_{X}$ and $\d{f_{!}} \colon \D_{Y} \to \D_{X}$. Furthermore, letting $\Delta
: Y \to Y \times_{X} Y$ denote the diagonal map induced by $f$, denote the
    Cartesian lifts of $\Delta$ by $\Delta^{*} \colon \C_{Y \times_{X} Y} \to
    \C_{Y}$ and $\d{\Delta^{*}} \colon \D_{Y \times_{X} Y} \to \D_{Y}$. Assume the
follwing hold:

\begin{itemize}
  \item $f$ is weakly ambidextrous with respect to $q$ and $p$.
    \item
        The square
        \[
        \begin{tikzcd}
            \D_{Y}
            &\D_{Y\times_{X}Y} \arrow[l, "\d{\Delta^{*}}"']
            \\\C_{Y}	\arrow[u, "F"']
            &\C_{Y\times_{X}Y} \arrow[l, "\Delta^{*}"'] \arrow[u, "F"']
        \end{tikzcd}
    \]
        satisfies the Beck-Chevalley condition.
\end{itemize}
In this situation, the diagram
\[
  \begin{tikzcd}
    \d{f^{*}} \d{f_{!}} F	\arrow[r, "\d{f^{*}} BC_{!}"] \arrow[d, "\nu_{\d{f}}"]
    &\d{f^{*}} F f_{!} \arrow[d,"\simeq"]
    \\
    F
    &F f^* f_{!} \arrow[l, "\nu_{f}"']
  \end{tikzcd}
\]
of natural transformations of functors in $\Fun(\C_{Y}, \D_{Y})$ commutes,
where $BC_{!}$ is the Beck-Chevalley transformation of left adjoints of the
diagram
\[
  \begin{tikzcd}
    \D_{Y}
    &\D_{X} \arrow[l, "\d{f^{*}}"']
    \\ \C_{Y}	\arrow[u, "F "']
    &\C_{X} \arrow[l, "f^{*}"'] \arrow[u, "F"']
  \end{tikzcd}
\]
induced by $F$.
\end{corollary}
\begin{proof}
  In the proof of \Cref{thm:norm.square} we proved commutativity of
  \cref{eq:norm.square.mate}, which is exactly what we want to show.
  When proving this we never used that $f^{*}$ and $\d{f^{*}}$ admit right
  adjoints.
\end{proof}

\begin{remark}[] \label{thm:less.conditions}
  In the proof of \Cref{thm:norm.square} and \Cref{thm:nu.square}, $q$ and $p$
  being Beck-Chevalley fibrations was a larger assumption than needed -- we
  only needed that \emph{some} Cartesian lifts admitted left adjoints, and that
  \emph{some} diagrams satisfied the Beck-Chevalley condition, with respect to
  both $q$ and $p$. To be precise, what we needed was the following:
  \begin{itemize}
    \item In the definition of weak ambidexterity of $f$, the fact that the pullback
      of $f$ along itself satisfies the Beck-Chevalley condition is needed to
      define $\nu_{f}$. Furthermore, it is also needed that the diagonal map
      $\Delta$ of $f$ is ambidextrous, hence that the pullback of $\Delta$
      along itself satisfies the Beck-Chevalley condition. Thus, we need that
      the pullback of $f$ along itself, the pullback of $\Delta$ along itself,
      the pullback of the diagonal of $\Delta$ along itself, and so on,
      satisfies the Beck-Chevalley condition. This is, however, only a finite
      number of diagrams, up to equivalence, as $f$ being weakly ambidextrous
      implies that $f$ is truncated. This collection of Beck-Chevalley
      transformations include $BC[\sigma]$ and $BC[\d{\sigma}]$, which we also
      directly needed to be equivalences in the proof of
      \Cref{thm:norm.square}.

    \item The Cartesian lifts with respect to $q$ and $p$ of $f, \pi_1, \Delta$, the diagonal of $\Delta$
        , and so on, all needed to admit left adjoints in the proof of
        \Cref{thm:norm.square} and the definition of (weak) ambidexterity of
        $f$. Again, this collection is only a finite number of morphisms, up to
        equivalence.
  \end{itemize}
  Thus, if these conditions are satisfied with respect to both $q$ and $p$, the
  maps $q$ and $p$ can just be assumed to be Cartesian fibrations, rather than
  Beck-Chevalley fibrations.
\end{remark}

%% file: normSquare/baseChange.tex
We show a useful way to construct new Beck-Chevalley fibrations from known
ones.
First, let $q$ be a Cartesian and coCartesian fibration, and let a diagram
\[
  \begin{tikzcd}
    & \C \arrow[d, "q"] \\
    \Y \arrow[r, "\alpha"] & \X
  \end{tikzcd}
\]
of $\infty$-categories be given. As described in \Cref{sec:straightening},
pulling $q$ back along $\alpha$ produces another Cartesian and coCartesian
fibration, which we denote by $p \colon \C_{\alpha} \to \Y$. There is a
canonical equivalence of fibers $(\C_{\alpha})_{Y} \simeq \C_{\alpha(Y)}$ for
any $Y \in \Y$, and the straightening of $p$ is equivalent to the functor
$\Y^{\op} \to \Cat_{\infty}$ induced by precomposing $\C_{(-)}$ with
$\alpha^{\op}$. Thus, for any diagram $\sigma$ in $\Y$, the lift of $\sigma$
by $p$ can be computed as the lift of $\alpha(\sigma)$ by $q$.

\begin{lemma}\label{thm:bc.pullback}
  Let $q \colon \C \to \X$ be a Beck-Chevalley fibration, let $\Y$ be an
  $\infty$-category that admits pullbacks, let $\alpha \colon \Y \to \X$ be a
  functor, and let $p \colon \C_{\alpha} \to \Y$ be the Cartesian and
  coCartesian fibration induced by pulling $q$ back along $\alpha$. If $\alpha$
  preserves pullbacks, then $p$ is a Beck-Chevalley fibration.
\end{lemma}

\begin{proof}
  We need to show  that given any pullback square
  \[
    \begin{tikzcd}
      \wt Y	\arrow[r, "h"] \arrow[d, "g_{Y}"] 	\arrow[dr, phantom, "\sigma"]
      &\wt X \arrow[d,"g_{X}"]
      \\
      Y	\arrow[r, "f"]
      &X
    \end{tikzcd}
  \]
  in $\Y$, the Beck-Chevalley transformation $BC[\sigma]$ induced by $p$ is an
  equivalence. But the Cartesian lift of $\sigma$ into $\C_{\alpha}$ is
  equivalent to the Cartesian lift of $\alpha(\sigma)$ into $\C$. As
  $\alpha$ preserves pullbacks and $q$ is a Beck-Chevalley fibration, it
  follows that $BC[\sigma]$ is an equivalence.
\end{proof}

\begin{definition}[]
  Let $q \colon \C \to \X$ be a Beck-Chevalley fibration, let $\Y$ be an
  $\infty$-category that admits pullbacks, and let $\alpha \colon \Y \to \X$ be
  a pullback preserving functor. By \Cref{thm:bc.pullback}, let $p \colon
  \C_{\alpha} \to \Y$ be the Beck-Chevalley fibration induced by the pullback
  of $q$ along $\alpha$. We refer to $p$ as the \emph{base change of $q$ along
  $\alpha$}.
\end{definition}

We show how ambidexterity is preserved under base change with the following two
results.

\begin{lemma} \label{thm:n-base.change}
  Let $q \colon \C \to \X$ be a Beck-Chevalley fibration, let $\Y$ be an
  $\infty$-category that admits pullbacks, let $\alpha \colon \Y \to
  \X$ be a pullback preserving functor, let $p \colon \C_{\alpha} \to \Y$
  be the base
  change of $q$ along $\alpha$, and let $f \colon Y \to X$ be a morphism in $\Y$.
  If $f$ is $n$-truncated and $\alpha f $ is (weakly) $n$-ambidextrous, then
  $f$ is (weakly) $n$-ambidextrous and $\nu_{f}^{n}$ is canonically equivalent to
  $\nu_{\alpha f}^{n}$.
\end{lemma}

\begin{proof}
  We induct on $n$. If $f$ is $(-2)$-truncated, then $f$ is an equivalence and
  hence (weakly) ambidextrous. Now, let $n\geq -1$ and assume the statement has been
  shown for any $(n-1)$-truncated map. First, assume that $\alpha f $ is weakly
  $n$-ambidextrous. The diagonal map $\Delta_{f}$ induced by $f$ is the unique
  map fitting into the diagram
  \[
    \begin{tikzcd}
      Y \arrow[drr, bend left, "\id"] \arrow[ddr, bend right,"\id"]
      \arrow[dr, "\Delta_{f}"]
      &&
      \\
      & Y \times_{X} Y \arrow[r, "\pi_{1}"] \arrow[d, "\pi_{2}"]
      & Y \arrow[d, "f"]\\
      &
      Y \arrow[r, "f"]
      &X
    \end{tikzcd},
  \]
  where the square is the pullback of $f$ along itself. Applying
  $\alpha$ to this diagram, as $\alpha$ preserves pullbacks, we see that
  $\alpha(\Delta_{f})$ is equivalent to $\Delta_{\alpha f }$, the diagonal
  map induced by $\alpha f $. As $\alpha f $ is weakly $n$-ambidextrous,
  $\Delta_{\alpha f }$ is $(n-1)$-ambidextrous. By the inductive hypothesis
  $\Delta_{f}$ is $(n-1)$-ambidextrous, hence $f$ is weakly $n$-ambidextrous.

  Now we show that $\nu_{f}^{n}$ and $\nu_{\alpha f}^{n}$ are canonically
  equivalent.
  First, note that the pullback square $\sigma$,
  \[
    \begin{tikzcd}
      Y \times_{X} Y	\arrow[r, "\pi_1"] \arrow[d, "\pi_2"] 	\arrow[dr,
      phantom, "\sigma"]
      &Y \arrow[d,"f"]
      \\
      Y	\arrow[r, "f"]
      &X
    \end{tikzcd},
  \]
 in $\Y$ given by pulling back $f$ along itself is by $\alpha$ send to the
 square $\alpha(\sigma)$, that is,
  \[
    \begin{tikzcd}
      \alpha(Y \times_{X} Y)	\arrow[r, "\alpha\pi_{1}"] \arrow[d, "\alpha
      \pi_{2}"] 	\arrow[dr, phantom, "\alpha(\sigma)"]
      &\alpha(Y) \arrow[d,"\alpha f"]
      \\
      \alpha(Y)	\arrow[r, "\alpha f"]
      &\alpha(X)
    \end{tikzcd}.
  \]
  As $\alpha$ preserves pullbacks, $\alpha(\sigma)$ must be canonically
  equivalent to the pullback square in $\X$ given by pulling $\alpha f$ back
  along itself. This, combined with the definition of the wrong way counit,
  shows that $\nu_{\alpha f}^{n}$ is given as the composite
  \[
    \nu_{\alpha f}^{n} \colon (\alpha f)^{*} (\alpha_{f})_{!}
    \xrightarrow{BC[\alpha(\sigma)]^{-1}} (\alpha \pi_1)_{!} (\alpha \pi_{2})^{*}
    \xrightarrow{(\alpha \pi_{1})_{!} \mu_{\alpha \Delta}^{n-1} (\alpha
    \pi_2)^{*}} (\alpha \pi_1)_{!} (\alpha \Delta)_{!} (\alpha \Delta)^{*}
    (\alpha \pi_2)^{*} \simeq \id_{\C_{\alpha(Y)}}.
  \]
  But this is equivalent to the definition of $\nu_{f}^{n}$ with respect to
  $p$, as any lift along $p$ can just be computed as the lift of the image under
  $\alpha$ by $q$. Thus, $\nu_{\alpha f}^{n}$ and $\nu_{f}^{n}$ are canonically
  equivalent.

  Now, assume that $\alpha f$ is $n$-ambidextrous. We need to show that
  given any pullback square
  \[
    \begin{tikzcd}
      \wt Y	\arrow[r, "h"] \arrow[d, "g_{Y}"]
      &\wt X \arrow[d,"g_{X}"]
      \\
      Y	\arrow[r, "f"]
      &X
    \end{tikzcd}
  \]
  in $\Y$, the morphism $h$ is weakly $n$-ambidextrous and has the property
  that $\smash{\nu_{h}^{n}}$ is the counit of an adjunction $h^{*} \dashv
  h_{!}$. The image of this
  square under $\alpha$ is the square
  \[
    \begin{tikzcd}
      \alpha \wt Y 	\arrow[r, "\alpha h "] \arrow[d, "\alpha g_{Y} "]
      &\alpha \wt X  \arrow[d,"\alpha g_{X} "]
      \\
      \alpha Y 	\arrow[r, "\alpha f "]
      &\alpha X
    \end{tikzcd},
  \]
  which is a pullback square as $\alpha$ preserves pullbacks. As $\alpha f $ is
  $n$-ambidextrous, it follows that $\alpha h$ is weakly ambidextrous and
  that $\nu_{\alpha f}^{n}$ is the counit of an adjunction $(\alpha f)^{*}
  \dashv (\alpha f)_{!}$. We have shown that $\alpha$ reflects weak
  $n$-ambidexterity, and that \smash{$\nu_{\alpha h}^{n}$} and
  \smash{$\nu_{h}^{n}$} are canonically equivalent. It follows that $h$
  is weakly $n$-ambidextrous, and that \smash{$\nu_{h}^{n}$} is the counit
  of an adjunction $h^{*} \dashv h_{!}$, hence $f$ is $n$-ambidextrous.
\end{proof}

\begin{proposition}[]\label{thm:base.change}
  Let $q \colon \C \to \X$ be a Beck-Chevalley fibration, $\Y$ an
  $\infty$-category that admits pullbacks, $\alpha \colon \Y \to \X$ a pullback
  preserving functor and $p \colon \C_{\alpha} \to \Y$ the base change of $q$
  along $\alpha$. Let $f \colon Y \to X$ be a morphism in $\Y$ and assume that
  $f$ is truncated. The following hold:
  \begin{enumerate}
    \item If $\alpha f$ is (weakly) ambidextrous, then $f$ is (weakly)
      ambidextrous and $\nu_{f}$ is canonically equivalent to $\nu_{\alpha f}$.

    \item If $\alpha f$ is weakly ambidextrous and $(\alpha f)^{*}$ admits a
      right adjoint $(\alpha f)_{*}$ then $f^{*}$ admits a right adjoint
      $f_{*}$, and the norm maps $\Nm_{f}$ and $\Nm_{\alpha f}$ are canonically
      equivalent.
  \end{enumerate}
\end{proposition}

\begin{proof}
  (1): For integers $m$ and $n$ with $-2 \leq m\leq n$, being
  $m$-truncated implies being $n$-truncated and by
  \Cref{thm:ambidex.coherence} (1) and (2), being (weakly) $m$-ambidextrous implies being
  (weakly) $n$-ambidextrous. We can thus find some integer $n \geq (-2)$ such
  that $f$ is $n$-truncated and $\alpha f $ is (weakly) $n$-ambidextrous. The
  result now follows from \Cref{thm:n-base.change}.

  (2): By (1), the morphism $f$ is weakly ambidextrous. As $f^{*}$ is
  canonically equivalent to $(\alpha f)^{*}$ it follows that $f^{*}$ admits a
  right adjoint $f_{*}$. The norm maps $\Nm_{f}$ and $\Nm_{\alpha f}$ are by
  definitions given as the mate of $\nu_{f}$ and $\nu_{\alpha f}$ under the
  adjunctions $f^{*} \dashv f_{*}$ and $(\alpha f)^{*} \dashv (\alpha f)_{*}$
  respectively. Both the wrong way counit maps and the adjunctions are
  equivalent under the same identifications, so the norm maps $\Nm_{f}$ and
  $\Nm_{\alpha f}$ must also be equivalent.
\end{proof}

\begin{remark}[]
  The condition that $f$ is truncated is necessary -- the constant map $\Y \to
  \X$ which sends every object of $\Y$ to a designated object $x \in \X$ and
  every
  morphism to the identity morphism trivially preserves pullbacks. But the
  constant map does not necessarily reflect ambidexterity, as not every
  morphism of $\Y$ needs to be truncated.
\end{remark}

Lastly, we show how natural transformations of functors relate to base change
of Beck-Chevalley fibrations.

\begin{proposition}\label{thm:bc.nat.trans}
  Let $q \colon \C \to \X$ be a Beck-Chevalley fibration, let $\Y$ be an
  $\infty$-category that admits pullbacks, let $\alpha, \beta \colon \Y \to \X$
  be pullback preserving functors, and let $F \colon \alpha \Longrightarrow
  \beta$ be a natural transformation. In this situation, $F$ induces a functor
  $F^{*}$ that preserves Cartesian edges and fits into the commutative diagram
  \[
    \begin{tikzcd}[column sep=normal]
      \C_{\beta} \arrow[dr, "p_{\beta}"'] \arrow[rr, "F^{*}"]
      && \C_{\alpha} \arrow[dl, "p_{\alpha}"] \\
      & \Y &
    \end{tikzcd}
  \]
  where $p_{\alpha}$ and $p_{\beta}$ are the base changes of $q$ along $\alpha$
  and $\beta$ respectively.
\end{proposition}

\begin{proof}
  By straightening $p_{\alpha}$ and $p_{\beta}$, we can summarize the data
  given with the diagram
  \[
    \begin{tikzcd}
      \Y^{\op} \arrow[r, bend left, "\alpha^{\op}" name=U]
      \arrow[r, bend right, "\beta^{\op}"' name=D]
      \arrow[Rightarrow, to path=(U)--(D)\tikztonodes, shorten=1.5ex, "F^{\op}"]
      & \X^{\op}  \arrow[r, "\C_{(-)}"]
      & \Cat_{\infty}
    \end{tikzcd}.
  \]
  The horizontal composite of $F^{\op}$ with the identity transformation of
  $\C_{(-)}$ yields a natural transformation $F^{*} \colon \C_{\beta(-)}
  \Longrightarrow \C_{\alpha(-)}$. As described in \Cref{sec:straightening},
  unstraightening this natural transformation yields a functor $F^{*} \colon
  \C_{\beta} \to \C_{\alpha}$ which preserves Cartesian edges and fits into the
  above diagram, as desired.
\end{proof}

%% file: applications/normApplications.tex
We show how \Cref{thm:norm.square} implies
Proposition 4.2.1 (1) of \cite{ambidexLurie} together with
Theorem 3.2.3 and Proposition 3.4.7 of \cite{1811}.

\subsection{Naturality of the Norm}
We have shown \Cref{thm:norm.square}, which shows a certain naturality property
of the norm map when comparing lifts along different Beck-Chevalley fibrations.
We now show how this theorem implies \cite[Proposition 4.2.1
(1)]{ambidexLurie}, which shows a similar naturality property of the norm map
when lifting along a single Beck-Chevalley fibration instead. We first show a
version of this result described in \cite[Remark 4.2.3]{ambidexLurie}, that is,
the situation where the lifts of the weakly ambidextrous maps also admit right
adjoints.
\begin{proposition}[{\cite[Remark 4.2.3]{ambidexLurie}}]
  \label{thm:4.2.1}
  Let $q \colon \C \to \X$ be a Beck-Chevalley fibration and let a
  commutative diagram $\tau$ be given in $\X$, depicted as
  \[
    \begin{tikzcd}
      \wt Y	\arrow[r, "h"] \arrow[d, "g_{Y}"] \arrow[dr, phantom, "\tau"]	&
      \wt X \arrow[d, "g_{X}"]
      \\
      Y	\arrow[r, "f"]							&		X
    \end{tikzcd}.
  \]
  Assume the following hold:
  \begin{itemize}
      \item The Cartesian lifts $f^{*}$ and $h^{*}$ admit right adjoints
          $f_{*}$ and $h_{*}$
      \item $f$ and $h$ are weakly ambidextrous.
      \item The square
          \[
         \begin{tikzcd}
             \wt Y \arrow[r, "\Delta_{h}"]\arrow[d, "g_{Y}"]
             &\wt Y \times_{\wt X} \wt Y   \arrow[d, ""]
             \\ Y	\arrow[r, "\Delta_{f}"]
             & Y \times_{X} Y
         \end{tikzcd}
         \]
         induced by the diagonal maps $\Delta_{f}$ and $\Delta_{h}$ satisfy
         the Beck-Chevalley condition.
  \end{itemize}
  In this situation, the diagram of natural transformations
  \[
    \begin{tikzcd}
      h_{!}g_{X}^{*}	\arrow[r, "BC_{!}{[\tau]}"] \arrow[d, "\Nm_{h}"] 	&
      g_{Y}^{*}f_{!}\arrow[d,"\Nm_{f}"] \\
      h_{*}g_{X}^{*}								&		g_{Y}^{*}f_{*} \arrow[l,
      "BC_{*}{[\tau]}"']
    \end{tikzcd}
  \]
  commutes.
\end{proposition}

\begin{proof}
    Consider the $\infty$-category $\Fun(\Delta^{1}, \X)$ of morphisms in $\X$.
    This category comes with two maps $\ev_0, \ev_1 \colon \Fun(\Delta^{1}, \X)
    \to \X$ given by evaluation at the initial and final vertex respectively,
    and with a natural transformation $F \colon \ev_0 \to \ev_1$ given by
    $F_{g} = g$ for any $g \in \Fun(\Delta^{1}, \X)$. Since limits of functors
    are computed pointwise, $\Fun(\Delta^{1}, \X)$ admit pullbacks, and the
    functors $\ev_{0}$ and $\ev_1$ preserve pullbacks. Denote by $p_0 \colon
    \C_{\ev_0} \to \Fun(\Delta^{1}, \X)$ and $p_1 \colon \C_{\ev_1} \to
    \Fun(\Delta^{1}, \X)$ the base changes of $q$ along $\ev_0$ and $\ev_1$
    respectively. By \Cref{thm:bc.nat.trans}, we obtain the commuative diagram
  \[
    \begin{tikzcd}[column sep=tiny]
      \C_{\ev_1} \arrow[rr, "F^{*}"] \arrow[dr, "p_{1}"']
      && \C_{\ev_0} \arrow[dl, "p_0"]
   \\ & \Fun(\Delta^{1}, \X) &
    \end{tikzcd}
  \]
  where $F^{*}$ preserves Cartesian edges and which for any morphism $g \in
  \Fun(\Delta^{1}, \X)$ satisfies that the restriction functor $F_{g}^{*}
  \colon (\C_{\ev_1})_{g} \to (\C_{\ev_0})_{g}$ is canonically equivalent to
  Cartesian lift $g^{*} \colon \C_{\ev_1(g)} \to \C_{\ev_0(g)}$ of $g$ by $q$.
  Note that a morphism in $\Fun(\Delta^{1}, \X)$ is simply a square $\Delta^{1}
  \times \Delta^{1} \to \X$. We can thus identify the square $\tau$ with a
  morphism $\tau \colon g_{X} \to g_{Y}$ in $\Fun(\Delta^{1}, \X)$ with
  $\ev_0(\tau) = h$ and $\ev_1(\tau) = f$. As $h$ and $f$ both are weakly
  ambidextrous, these maps are by \Cref{thm:ambidex.truncated} both truncated,
  hence $\tau$ is truncated. Thus, by \Cref{thm:base.change} (1) $\tau$ is
  weakly ambidextrous with respect to both $p_{0}$ and $p_{1}$. Furthermore, by
  \Cref{thm:base.change} (2) the norm map $\Nm_{\tau}$ with respect to $p_{1}$
  exists and is equivalent to $\Nm_{f}$ and the norm map $\Nm_{\tau}$ with
  respect to $p_{0}$ exists, and is equivalent to $\Nm_{h}$. Finally, the
  square induced by the lifts of the diagonal map $\Delta_{\tau}$ which is
  required to satisfy the Beck-Chevalley condition for us to apply
  \Cref{thm:norm.square}, is equivalent to the lift of the square
\[
    \begin{tikzcd}
             \wt Y \arrow[r, "\Delta_{h}"]\arrow[d, "g_{Y}"]
             &\wt Y \times_{\wt X} \wt Y   \arrow[d, ""]
             \\ Y	\arrow[r, "\Delta_{f}"]
             & Y \times_{X} Y
         \end{tikzcd}
\]
along $q$. The result now follows from \Cref{thm:norm.square}.
\end{proof}

The actual version of \cite[Proposition 4.2.1 (1)]{ambidexLurie},
where $f^{*}$ and $h^{*}$ are not assumed to admit right adjoints, is the
following:
\begin{proposition}[{\cite[Proposition 4.2.1 (1)]{ambidexLurie}}]
  Let $q \colon \C \to \X$ be a Beck-Chevalley fibration, and let a pullback
  square $\tau$ be given in $\X$, depicted as
    \[
    \begin{tikzcd}
      \wt Y	\arrow[r, "h"] \arrow[d, "g_{Y}"] \arrow[dr, phantom, "\tau"] 	&
      \wt X \arrow[d,"g_{X}"] \\ Y	\arrow[r, "f"]							&		X
    \end{tikzcd}.
  \]
  If $f$ is weakly ambidextrous, then the diagram of
  natural transformations
   \[
    \begin{tikzcd}
      h^{*} h_{!} g_{X}^{*}	\arrow[r, "BC{[\tau]}"] \arrow[d, "\nu_{h}"]
      &h^{*} g_{Y}^{*}f_{!} \arrow[d,"\simeq"]
      \\
      g_{X}^{*}
      &g_{X}^{*}f^{*}f_{!} \arrow[l, "\nu_{f}"']
    \end{tikzcd}
  \]
  commutes.
\end{proposition}

\begin{proof}
    First note that $h$ is weakly ambidextrous by \Cref{thm:ambidex.properties}
    (4). Futhermore, as mentioned in \Cref{thm:ambidex.properties} (4), since
    $\tau$ is a pullback square, the square
\[
    \begin{tikzcd} \wt Y \arrow[r, "\Delta_{h}"]\arrow[d, "g_{Y}"] &\wt Y
        \times_{\wt X} \wt Y   \arrow[d, ""] \\ Y	\arrow[r, "\Delta_{f}"] & Y
        \times_{X} Y
         \end{tikzcd}
  \]
  is also a pullback square, and thus satisfies the Beck-Chevalley condition.
  The proof is now the same as the proof of \Cref{thm:4.2.1}, except we invoke
  \Cref{thm:nu.square} at the end rather than \Cref{thm:norm.square}.
\end{proof}

\begin{remark}[]
  \Cref{thm:4.2.1} does partially imply \Cref{thm:norm.square}. To see
  this, let a commutative diagram
  \[
    \begin{tikzcd}[column sep=normal]
      \C \arrow[rr, "F"]  \arrow[dr, "q"']
      && \D \arrow[dl, "p"] \\
      &\X&
    \end{tikzcd}
  \]
  of $\infty$-categories be given as in \Cref{thm:norm.square}, that is, $q$
  and $p$ are Beck-Chevalley fibrations, and $F$ is a functor that preserves
  Cartesian edges. Furthermore, let a morphism $f \colon Y \to X$ in $\X$ be
  given, which is weakly ambidextrous with respect to both $q$ and $p$, and
  such that $f^{*}$ in $\C$ and $\d{f^{*}}$ in $\D$ both admit right adjoints.
  By straightening, as described in \Cref{sec:straightening}, this data can
  equivalently be described by the diagram
  \[
    \begin{tikzcd}
      \X^{\op} \arrow[r, bend left, "\C_{(-)}" name=U]
      \arrow[r, bend right, "\D_{(-)}"' name=D]
      \arrow[Rightarrow, to path=(U)--(D)\tikztonodes, shorten=1.5ex, "F"]
      & \Cat_{\infty}
\end{tikzcd}
  \]
  of functors into $\Cat_{\infty}$, with a natural transformation between them.
  This information can be encoded as a single functor
  \[
      \begin{tikzcd}
        (\X \times \Delta^{1})^{\op} \arrow[r, "H"] & \Cat_{\infty}
      \end{tikzcd}
  \]
  which restricted to $\X^{\op} \times \{1\}$ is $\C_{(-)}$, restricted to
  $\X^{\op} \times \{0\}$ is $\D_{(-)}$, and sends any morphism of the form
  $(X, 1) \to (X, 0)$ to the functor $F_{X} \colon \C_{X} \to \D_{Y}$ provided
  by the natural transformation $F$. By unstraightening we thus obtain a
  Cartesian fibration over $\X \times \Delta^{1}$. By checking that the
  necessary squares satisfy the Beck-Chevalley condition, we can with an analog
  of \Cref{thm:less.conditions} apply \Cref{thm:4.2.1} to the pullback square
  \[
    \begin{tikzcd}
      (Y, 1)	\arrow[r, "{(f,1)}"] \arrow[d, ""]
      &(X, 1) \arrow[d,""]
      \\
      (Y, 0)	\arrow[r, "{(f,0)}"]
      &(X, 0)
    \end{tikzcd}
  \]
  in $\X \times \Delta^{1}$ to obtain \Cref{thm:norm.square}.
\end{remark}

\subsection{The Norm for Local Systems}
We show how \Cref{thm:norm.square} implies \cite[Theorem 3.2.3]{1811}, showing
naturality of the norm map with respect to Beck-Chevalley fibrations of local
systems. Recall from \Cref{thm:local.systems} that for $\C$ an
$\infty$-category, there is a Cartesian fibration $q \colon \LocSys(\C) \to \S$
from the $\infty$-category of $\C$-valued local systems to the
$\infty$-category of spaces. This Cartesian fibration is obtained by
unstraightening the functor $\Fun(-, \C) \colon \S^{\op} \to \Cat_{\infty}$,
and thus any fiber $\LocSys(\C)_{A}$ is canonically equivalent to the functor
category $\Fun(A, \C)$, and any Cartesian lift of a morphism $f \colon B \to A$
in $\S$ can be identified with the pre-composition functor $f^{*} \colon
\Fun(A, \C) \to \Fun(B, \C)$. Furthermore, if $\C$ admits colimits, then $q$ is
also a coCartesian fibration, as pre-composition has a left adjoint given by
left Kan extension. We have the following:

\begin{proposition}[4.3.3 Lurie] \label{thm:4.3.3.lurie}
  Let $\C$ be an $\infty$-category which admits colimits. In this situation, the
  forgetful functor
  \[
    \LocSys(\C) \to \S
  \]
  is a Beck-Chevalley fibration.
\end{proposition}

\begin{proof}
  This is \cite[Proposition 4.3.3]{ambidexLurie} -- we will prove
  \Cref{thm:bc.condition} below, which is a more general statement.
\end{proof}

Given a functor $F \colon \C \to \D$, we obtain a functor $F_{*} \colon
\LocSys(\C) \to \LocSys(\D)$, induced by the natural transformation $\Fun(-,
\C) \to \Fun(-, \D)$ given by post-composition with $F$. As post-composition
commutes with pre-composition, it follows that $F_{*}$ preserves Cartesian
edges. We first show the following weaker version of \cite[Theorem
3.2.3]{1811}:

\begin{proposition} \label{thm:1811.1}
  Let $F \colon \C \to \D$ be a functor between $\infty$-categories that admit
  colimits and denote by $q$ and $p$ the Beck-Chevalley fibrations $q \colon
  \LocSys(\C) \to \S$ and $p \colon \LocSys(\D) \to \S$. Let $f \colon B \to A$
  be a map of spaces, denote the Cartesian lifts of $f$ by $f^{*} \colon
  \Fun(A, \C) \to \Fun(B, \C)$ and $\d{f^{*}} \colon \Fun(A, \D) \to \Fun(B,
  \D)$ respectively and denote the coCartesians lifts of $f$ by $f_{!} \colon
  \Fun(B, \C) \to \Fun(A, \C)$ and $\d{f_{!}} \colon \Fun(B, \D) \to \Fun(A,
  \D)$ respectively. Furthermore, letting $\Delta \colon B \to B \times_{A} B$
  denote the diagonal map induced by $f$, denote the Cartesian lifts of
  $\Delta$ by $\Delta^{*} \colon \Fun(B \times_{A} B, \C) \to \Fun(B, \C)$ and
  $\d{\Delta^{*}} \colon \Fun(B\times_{A}B , \D) \to \Fun(B, \D)$. Assume the
  following hold:
  \begin{itemize}
      \item $f^{*}$ and $\d{f^{*}}$ admit right adjoints $f_{*}$ and
          $\d{f_{*}}$ respectively.

      \item $f$ is weakly ambidextrous with respect to $q$ and $p$.

      \item The square
          \[
              \begin{tikzcd}
                  \Fun(B, \D) &\Fun(B \times_{A} B, \D) \arrow[l,
                  "\d{\Delta^{*}}"']
                  \\ \Fun(B, \C)	\arrow[u, "F_{*}"'] &\Fun(B\times_{A} B,
                  \C) \arrow[l, "\Delta^{*}"'] \arrow[u, "F_{*}"']
              \end{tikzcd}
          \]
          satisfies the Beck-Chevalley condition.

  \end{itemize}
  In this situation, the diagram
  \[
     \begin{tikzcd}
       \d{f_{!}}F_{*}	\arrow[r, "BC_{!}"] \arrow[d, "\Nm_{\d{f}}"]
       &F_{*}f_{!} \arrow[d,"\Nm_{f}"]
       \\
       \d{f_{*}}F_{*}
       &F_{*}f_{*}\arrow[l, "BC_{*}"']
     \end{tikzcd}
  \]
  commutes, where $BC_{!}$ and $BC_{*}$ are the two Beck-Chevalley
  transformations of the diagram
  \[
    \begin{tikzcd}
      \Fun(B, \D)
      &\Fun(A, \D) \arrow[l, "\d{f^{*}}"']
      \\\Fun(B, \C)	\arrow[u, "F_{*}"']
      &\Fun(A, \C) \arrow[l, "f^{*}"'] \arrow[u, "F_{*}"']
    \end{tikzcd}.
  \]
\end{proposition}

\begin{proof}
  The maps $q \colon \LocSys(\C) \to \X$ and $p \colon \LocSys(\D) \to \X$ are
  in fact Beck-Chevalley fibrations, by \Cref{thm:4.3.3.lurie}. By applying
  \Cref{thm:norm.square}, the result now follows.
\end{proof}
We are interested in a slightly stronger version of this statement, where $\C$
and $\D$ are not assumed to admit all colimits, so $\LocSys(\C) \to \S$ and
$\LocSys(\D) \to \S)$ might not be Beck-Chevalley fibrations. One way of doing
this is by considering \Cref{thm:less.conditions} and keeping track of what
limits and colimits we need. We will do this another way, by constructing a
fully faithful functor $j \colon \C \to \wt \C$ into an $\infty$-category $\wt
\C$ which admits all colimits so $\LocSys(\wt \C) \to \S$ is a
Beck-Chevalley fibration. First, we introduce the following terminology from
\cite{1811}:

\begin{definition}[]
  Let $f \colon B \to A$ be a map of spaces, and let $\C$ be an
  $\infty$-category. We say that $\C$ admits \emph{$f$-limits}, respectively
  \emph{$f$-colimits}, if $\C$ admits limits, respectively colimits, of shape
  $f^{-1}\{a\}$ for all $a \in A$.
\end{definition}

This terminology is relevant in the following sense:

\begin{lemma}[] \label{thm:kan.extension}
Let $f\colon B \to A$ be a map of spaces, let $\C$ be an $\infty$-category, and
let $f^{*} \colon \Fun(A, \C) \to \Fun(B, \C)$ be the induced pre-composition
functor. If $\C$ admits $f$-colimits, then $f^{*}$ admits a left adjoint
$f_{!}$. Likewise, if $\C$ admits $f$-limits, then $f^{*}$ admits a right
adjoint $f_{*}$.
\end{lemma}

\begin{proof}
If $\C$ has sufficient colimits, then we can compute the left adjoint $f_{!}$
of $f^{*}$ pointwise, by left Kan extension: Given a functor $F \in \Fun(B,
\C)$ computing the colimit
\[
  f_{!}F (a) \simeq \colim_{A_{/a} \times_{A} B} F,
\]
for each $a \in A$ computes $f_{!}F$, where $F \colon A_{/a} \times_{A} B \to
\C$ is the diagram given by $(f(b) \to a) \mapsto F(b)$. As $A$ is a space $A$
is an $\infty$-groupoid, so the space $A_{/a}$ is contractible. Thus, this
diagram is canonically equivalent to the diagram $F_{|f^{-1}\{a\}} \colon
f^{-1}\{a\} \to \C$ given by restriction of $F$. It follows that left Kan
extension produces a left adjoint $f_{!}$ to $f^{*}$ if $\C$ admits
$f$-colimits. A similar argument, using right Kan extension instead of
left, shows that that $f^{*}$ admits a right adjoint
$f_{*}$ if $\C$ admits $f$-limits.
\end{proof}

\begin{remark}[] \label{thm:projection.fiber}
  Given a map of spaces $f \colon B \to A$, consider the pasting of pullback
  diagrams,
  \[
    \begin{tikzcd}
      \pi_2^{-1}\{b\} \arrow[r, ""] \arrow[d, ""]
      & B \times_{A} B \arrow[r, "\pi_{1}"] \arrow[d, "\pi_{2}"]
      &  B \arrow[d, "f"]
      \\ \{b\} \arrow[r, hookrightarrow, ""]
      & B \arrow[r, "f"]
      & A
    \end{tikzcd}.
  \]
  From this it follows that the fibers of the projection maps $\pi_1$ and
  $\pi_2$ are canonically equivalent to the non-empty fibers of $f$. Thus,
  $\pi_{1}^{*}$ and $\pi_{2}^{*}$ also admit left and right adjoints if $\C$
  admits $f$-colimits and $f$-limits respectively.
\end{remark}

\Cref{thm:kan.extension} leads to the following generalization of
\Cref{thm:4.3.3.lurie}. Our proof is inspired by \cite[Lemma 6.1.6.3]{ha}:
\begin{proposition}[]\label{thm:bc.condition}
  Let $\C$ an $\infty$-category and let a pullback square
  \[
    \begin{tikzcd}
      \wt B	\arrow[r, "h"] \arrow[d, "g_{B}"]
      &\wt A \arrow[d,"g_{A}"]
      \\
      B	\arrow[r, "f"]
      &A
    \end{tikzcd}
  \]
  in $\S$ be given. Let $\sigma$ denote the induced diagram
  \[
    \begin{tikzcd}
      \Fun(\wt B, \C)	  	\arrow[dr, phantom, "\sigma"]
      &\Fun(\wt A, \C) \arrow[l, "h^{*}"']
      \\\Fun(B, \C)	\arrow[u, "g_{B}^{*}"']
      &\Fun(A, \C) \arrow[l, "f^{*}"'] \arrow[u, "g_{A}^{*}"']
    \end{tikzcd}
  \]
  of local systems. In this situation, the following hold:
  \begin{enumerate}
    \item If $\C$ admits $f$-colimits, then the Beck-Chevalley transformation
      $BC_{!}[\sigma] \colon h_{!} g_{B}^{*} \to g_{A}^{*}f_{!}$ exists and is
      an equivalence, that is, $\sigma$ satisfies the Beck-Chevalley condition.

    \item If $\C$ admits $f$-limits, then the Beck-Chevalley transformation
      $BC_{*}[\sigma] \colon \wt g_{A}^{*}f_{*} \to f_{*} g_{B}^{*}$ exists
      and is an equivalence.
  \end{enumerate}
\end{proposition}
\begin{proof}
  We show (1). Statement (2) can be shown by a similar argument using right Kan extension
  rather than left. The existence of a left adjoint $f_{!}$ to $f^{*}$ follows
  from \Cref{thm:kan.extension}. Furthermore, as $h$ is the pullback of $f$,
  $h^{*}$ also admits a left adjoint $f_{!}$. Thus, the Beck-Chevalley
  transformation $BC_{!}[\sigma] \colon h_{!} g_{B}^{*} \to g_{A}^{*}f_{!}$
  exists. As we already have a natural transformation, we just need to show it
  is pointwise an equivalence. That is, for each $F \in \Fun(B, \C)$ we must
  show that $BC_{!}[\sigma]_{F} \colon h_{!} g_{B}^{*}F \to g_{A}^{*}f_{!}F$ is
  an equivalence of functors in $\Fun(\wt A, \C)$. A natural transformation of
  functors being an equivalence just means that it is an equivalence at each point.
  We can thus reduce to the case where $\wt A = \{*\}$, the terminal space.
  That is, we have to show that for the pullback diagram
  \[
    \begin{tikzcd}
      B_{a}	\arrow[r, "p"] \arrow[d, hookrightarrow, "j"]
      &\{a\} \arrow[d, hookrightarrow, "i"]
      \\
      B	\arrow[r, "f"]
      &A
    \end{tikzcd},
  \]
  the Beck-Chevalley transformation $p_{!}j^{*} \to i^{*}f_{!}$ is an
  equivalence. Given a functor $F \in \Fun(B, \C)$, the left Kan extension
  $p_{!}F$ is given by computing the colimit of $F$, as $p$ is the projection
  to a point. So,
  \[
    p_{!}j^{*} (F) (a) \simeq p_{!} (Fj) (a) \simeq \colim_{B_{a}} F_{|B_{a}}.
  \]
  For the other composite, we have
  \[
    i^{*}f_{!}(F)(a) \simeq (f_{!}F) (a) \simeq \colim_{A_{/a} \times_{A} B} F,
  \]
  by the formula for pointwise left Kan extension. These two colimits both exist
  and are canonically equivalent, as is described in the proof of
  \Cref{thm:kan.extension}. Thus, $BC_{!}[\sigma]$ is an equivalence.
\end{proof}

We will need the following Lemma to prove \cite[Theorem 3.2.3]{1811}:
\begin{lemma}[]\label{thm:delta.bc}
    Let $F \colon \C \to \D$ be a functor of $\infty$-categories and $f \colon
    B \to A$ a map of spaces. If $\C$ and $\D$  admit, and $F$ preserves, all
    $f$-colimits, then the square
    \[
        \begin{tikzcd}
            \Fun(B, \D)
            &\Fun(A, \D) \arrow[l, "\d{f^{*}}"']
            \\\Fun(B, \C)	\arrow[u, "F_{*}"']
            &\Fun(A, \C) \arrow[l, "f^{*}"'] \arrow[u, "F_{*}"']
        \end{tikzcd}
    \]
    satisfies Beck-Chevalley condition.
\end{lemma}
\begin{proof}
    See \cite[Proposition
    3.2.2]{1811} -- the proof follows from the
    point-wise formulas for left Kan extension.
\end{proof}

We now state the definition of ambidexterity of maps of spaces as it is given
in \cite[Definition 3.1.5]{1811}. Note that \Cref{thm:bc.condition} is needed
in this construction, in the same way that we needed certain squares to satisfy
the Beck-Chevalley condition in \Cref{def:n-ambidex}.
\begin{definition}[] \label{def:space.ambidex}
  Let $\C$ be an $\infty$-category and $n\geq -2$ an integer. A map of spaces
  $f \colon B \to A$ is called
  \begin{enumerate}
    \item \emph{weakly $n$-$\C$-ambidextrous} if $f$ is $n$-truncated, $\C$
      admits $f$-limits and $f$-colimits, and either of two holds:
      \begin{itemize}
        \item $n=-2$, in which case the inverse of $f^{*}$ is both a left and
          right adjoint of $f^{*}$. We define the \emph{canonical norm} map on
          $f^{*}$ to be the identity of some inverse of $f^{*}$.
        \item $n\geq-1$, and the diagonal $\Delta \colon B \to B\times_{A}B$ of
          $f$ is $(n-1)$-$\C$-ambidextrous. In this case we define the
          \emph{canonical norm} on $f^{*}$ to be the diagonally induced one
          from the canonical norm of $\Delta$ (in the sense of
          \Cref{def:n-ambidex}).
      \end{itemize}
    \item \emph{$n$-$\C$-ambidextrous}, if $f$ is weakly $n$-$\C$-ambidextrous
      and its canonical norm map is an equivalence.
  \end{enumerate}
A map $f\colon B \to A$ of spaces is called (weakly) $\C$-ambidextrous, if it
is (weakly) $n$-$\C$-ambidextrous for some $n\geq -2$.
\end{definition}

\begin{remark}[]
  For a map of spaces $f \colon B \to A$, \Cref{def:space.ambidex} differs from
  \Cref{def:ambidex} with respect to a Beck-Chevalley fibration $\LocSys(\C)
  \to \S$ in the following ways:
\begin{itemize}
  \item In \Cref{def:ambidex}, we did not assume that $f^{*}$ admitted a right
    adjoint -- in \Cref{def:space.ambidex}, $f$ being weakly ambidextrous
    includes $\C$ admitting $f$--limits, which by \Cref{thm:kan.extension}
    implies that $f^{*}$ admits a right adjoint $f_{*}$.

  \item In \Cref{def:ambidex} we needed that $\LocSys(\C) \to \S$ was a
    Beck-Chevalley fibration, which happens when $\C$ admits all colimits. In
    \Cref{def:space.ambidex}, $\C$ is only assumed to admit
    $f$-colimits, so $\LocSys(\C) \to \S$ is not necessarily
    a Beck-Chevalley fibration.
\end{itemize}
If $\C$ admits all limits and colimits, then
\Cref{def:ambidex} with respect to $\LocSys(\C) \to \S$  and
\Cref{def:space.ambidex} coincide.
\end{remark}

The following lemma shows how we can generalize from the situation where $\C$
only admits certain colimits to the situation where $\LocSys(\C) \to \S$ is
a Beck-Chevalley fibration.

\begin{lemma} \label{thm:ff.colimits}
  Let $\C$ be an $\infty$-category and let $f \colon B \to A$ be a weakly
  $\C$-ambidextrous map of spaces. In this situation, there exists an
  $\infty$-category $\wt \C$ and a fully faithful functor $j \colon \C \to \wt
  \C$ with the following properties:

  \begin{enumerate}
    \item The induced map $q \colon \LocSys(\wt \C) \to \S$ is a Beck-Chevalley
      fibration, and $f$ is weakly ambidextrous with respect to $q$.

    \item The Cartesian lift $\d{f^{*}}$ of $f$ in $\LocSys(\wt \C)$ admits a
      right adjoint $\d{f_{*}}$ in addition to the left adjoint $\d{f_{!}}$.

    \item Letting $j_{*} \colon \LocSys(\C) \to \LocSys(\wt \C)$ denote the
      functor induced by postcomposing, the functor $j_{*}$ sends the triple
      $f_{!} \dashv f^{*} \dashv f_{*}$ in $\LocSys(\C)$ to the triple
      $\d{f_{!}} \dashv \d{f^{*}} \dashv \d{f_{*}}$ in $\LocSys(\wt \C)$.

    \item The functor $j_{*}$ sends the wrong-way counit $\nu_{f}$ and norm map
      $\Nm_{f}$ in \smash{$\LocSys(\C)$} to the wrong-way counit $\nu_{\d{f}}$ and norm
      map $\Nm_{\d{f}}$ in \smash{$\LocSys(\wt \C)$.}
  \end{enumerate}
\end{lemma}
\begin{proof}
  By \cite[Proposition 5.3.6.2]{htt} it holds that for any collection $\mathcal
  R$ of colimit diagrams in $\C$, there is a fully faithful functor $j \colon
  \C \to \P_{\mathcal R}(\C)$ which preserves colimits in $\mathcal R$, and
  where $\P_{\mathcal R}(\C)$ admits all colimits. Furthemore, the functor $j$
  preserves all limits that exist in $\C$: The $\infty$-category $\P_{\R}(\C)$
  is defined as the full subcategory of $\P(\C)$ which sends diagrams in
  $\R^{\op}$ to limit diagrams in $\S$. The inclusion $i \colon \P_{\R} (\C)
  \to \P(\C)$ then preserves limits, as it is defined by requiring certain
  diagrams are send to limit diagrams, and limits commute with limits. The
  Yoneda embedding $\C \to \P(\C)$ also preserves limits, by \cite[Proposition
  5.1.3.2]{htt}, so $j$ preserves limits.

  Let $\R$ be the collection of all colimit diagrams in $\C$ indexed by small
  truncated spaces, and define $\wt \C := \P_{\R}(\C)$. In particular, for any
  weakly $\C$-ambidextrous map $f$, $\R$ includes all colimit diagrams indexed
  by fibers of $f$, as $f$ being weakly ambidextrous implies $f$ is truncated
  by \Cref{thm:ambidex.truncated}. Thus, for any weakly ambidextrous map $f$,
  the functor $j$ preserves all $f$-colimits.
  Therefore, the induced post-composition functor $j_{*} \colon \LocSys(\C) \to
  \LocSys(\wt \C)$ sends the triple $f_{!} \dashv f^{*} \dashv f_{*}$ to
  $\d{f_{!}} \dashv \d{f^{*}} \dashv \d{f_{*}}$, and likewise for $\Delta$ and
  $\pi$. This ensures that $f$ is weakly ambidextrous with respect to
  $\LocSys(\wt \C) \to \S$, and that the wrong-way counit $\nu_{f}$ and norm
  map $\Nm_{f}$ are send to $\nu_{\d{f}}$ and $\Nm_{\d{f}}$ respectively. As
  $j$ is fully faithful, any properties we show for the lifts of $f$ into
  $\LocSys({}\wt \C{})$ also hold for the lifts into $\LocSys(\C)$, by
  restricting the functor category $\Fun(X, \wt \C)$ to those functors that
  factor through $j$. Thus, (1)-(4) hold.
\end{proof}

We are now ready to prove \cite[Theorem 3.2.3]{1811}:
\begin{corollary}[{\cite[Theorem 3.2.3]{1811}}]\label{thm:1811.1.2}
  Let $f \colon B \to A$ be a map of spaces and denote by $\Delta \colon B \to
  B\times_{A}B$ the induced diagonal map. Let $F \colon \C \to \D$ be a
  functor of $\infty$-categories which preserves $\Delta$-colimits, and
  denote by $F_{*}$ the induced post-composition functor $F_{*} \colon
  \LocSys(\C) \to \LocSys(\D)$. If $f$ is weakly $\C$-ambidextrous and weakly
  $\D$-ambidextrous, then the diagram
  \[
    \begin{tikzcd}
      \d{f_{!}} F_{*} \arrow[r, "BC_{!}"] \arrow[d, "\Nm_{\d{f}}"]
      & F_{*}f_{!} \arrow[d, "\Nm_{f}"]\\
      \d{f_{*}}F_{*}
      & F_{*}f_{*} \arrow[l, "BC_{*}"']
    \end{tikzcd}
  \]
  commutes, where $BC_{!}$ and $BC_{*}$ are the Beck-Chevalley transformations
  of left and right adjoints, respectively, of the diagram
  \[
    \begin{tikzcd}
      \Fun(B, \D)
      &\Fun(A, \D) \arrow[l, "\d{f^{*}}"']
      \\\Fun(B, \C)	\arrow[u, "F_{*}"']
      &\Fun(A, \C) \arrow[l, "f^{*}"'] \arrow[u, "F_{*}"']
    \end{tikzcd}.
  \]
\end{corollary}
\begin{proof}
  By \Cref{thm:ff.colimits}, we obtain $\infty$-categories $\wt{\C}$ and
  $\wt{\D}$ which admit all colimits, and fully faithful inclusions $j_{\C}
  \colon \C \to \wt{\C}$ and $j_{\D} \colon \D \to \wt{\D}$. By left Kan
  extension of $j_{\D}\circ F$ along $j_{\C}$, which exists as $\wt \D$ admits
  colimits, we obtain a functor $\wt{F} \colon \wt{\C} \to \wt{\D}$ which fits
  into the commutative diagram
  \[
    \begin{tikzcd}
      \C	\arrow[r, "F"] \arrow[d, hookrightarrow, "j_{\C}"]
      &\D \arrow[d, hookrightarrow, "j_{\D}"]
      \\
      \wt{\C}	\arrow[r, "\wt F"]
      &\wt{\D}
    \end{tikzcd}.
  \]
  Applying \Cref{thm:delta.bc} to $\Delta$, the result now follows from
  \Cref{thm:1811.1} by replacing $\LocSys(\C)$ with $\LocSys(\wt \C)$,
  $\LocSys(\D)$ with $\LocSys(\wt \D)$, and $F$ with $\wt F$.
\end{proof}

\begin{remark}[]
    As presented in \cite[Theorem 3.2.3]{1811}, the assumptions differ from
    \Cref{thm:1811.1.2} in two ways: $F$ is assumed to preserve $(m-1)$-finite
    colimits rather than $\Delta$-colimits and $f$ is assumed to be an
    $m$-finite map of spaces in addition to being weakly ambidextrous. A
    generalization of \cite[Lemma 5.5.6.15]{htt} shows that $f$ being
    $m$-finite implies that $\Delta$ is $(m-1)$-finite. Thus, in this situation
    $F$ also preserves $\Delta$-colimits, so \cite[Theorem 3.2.3]{1811} does in
    fact follow from \Cref{thm:1811.1.2}.
\end{remark}

\subsection{The Norm for Equivariant Powers}
\label{ssub:prop_3_4_7}
We show how \Cref{thm:norm.square} implies another result from \cite{1811},
namely \cite[Proposition 3.4.7]{1811}. We start by introducing the relevant
terminology. In the following, for $p$ a prime, we denote the \emph{symmetric
group of degree $p$} by $\Sigma_{p}$, and the \emph{cyclic group of order $p$}
by $C_{p}$, which is a subgroup of $\Sigma_{p}$. Let $\C$ be a symmetric
monoidal $\infty$-category and $p$ a prime. In this situation, there is a
functor $\Theta^{p} \colon \C \to \Fun(BC_{p}, \C)$, with $BC_{p}$ the
classifying space of $C_{p}$. This functor has the property that
pre-composition of $\Theta^{p}$ with $e \colon \{*\} \to BC_{p}$ is homotopic
to the $p$-th power functor $(-)^{\otimes p} \colon \C \to \C$, which on
morphisms is given by sending $f \colon Y \to X$ to $f^{\otimes p} \colon
Y^{\otimes p} \to X^{\otimes p}$.

Given a space $A \in \S$, we define the $C_{p}$-equivariant $p$-power of $A$ as
the space $A^{p}_{hC_{p}} = (A^{p} \times EC_{p})/C_{p}$. This
category is a model of the $\infty$-categorical quotient of $A^{p}$ by $C_{p}$.
We will denote this quotient $(A^{p})_{hC_{p}}$ by $A \wr C_{p}$.

We start by stating \cite[Lemma 3.4.1]{1811}:
\begin{lemma}[]\label{thm:theta.pb.preserving}
  The functor $(-)\wr C_{p} \colon \S \to \S$ preserves pullbacks.
\end{lemma}
\begin{proof}
  Letting $e \colon \{*\} \to BC_{p}$ denote a choice of base point, the
  functor $(-)\wr C_{p}$ can be identified with the composite
  \[
    \S \xrightarrow{e_{*}} \Fun(BC_{p}, \S) \simeq \S_{/BC_{p}}
    \xrightarrow{\pi} \S.
  \]
  The functor $e_{*}$ is right adjoint to the pre-composition functor $e^{*}$
  and hence preserves all limits. The canonical projection $\pi$ is a right
  fibration by the dual of \cite[Corollary 2.1.2.2]{htt}. Thus, by the dual of
  \cite[Proposition 4.4.2.9]{htt} it preserves limits of contractible shape
  which includes pullbacks.
\end{proof}

\begin{definition}[]
  Let $\C$ be a symmetric monoidal $\infty$-category. For each space $A \in \S$
  we define the functor
  \[
    \Theta^{p}_{A} \colon \Fun(A, \C) \to \Fun(A \wr C_{p}, \C)
  \]
  to be the composition of $(-)^{p}_{hC_{p}}$ with
  \[
    (\C^{p})_{hC_{p}} \to (\C^{p})_{h\Sigma_{p}} \xrightarrow{\otimes} \C.
  \]
  If the space $A$ is clear from the context we supress the subscript $A$. Note
  that this definition is natural in $A$, that is, $\Theta^{p}$ is a natural
  transformation from $\Fun(-, \C)$ to $\Fun( (-)\wr C_{p}, \C)$.
\end{definition}

To prove \Cref{thm:3.4.7}, we will need the following lemma.
\begin{lemma}[] \label{thm:theta.bc}
    Let $f \colon B \to A$ be a map of spaces and let $(\C, \otimes, \boldone)$
    be a symmetric monoidal category that admits all $f$-colimits. If $\otimes$
    distributes over all $f$-colimits, then the square
    \[
        \begin{tikzcd}
            \Fun(B \wr C_{p}, \C)
            &\Fun(A \wr \C_{p}, \C) \arrow[l, "(f\wr C_{p})^{*}"']
            \\\Fun(B, \C)	\arrow[u, "\Theta^{p}"']
            &\Fun(A, \C) \arrow[l, "f^{*}"'] \arrow[u, "\Theta^{p}"']
        \end{tikzcd}
    \]
    satisfies the Beck-Chevallley condition.
\end{lemma}
\begin{proof}
    Omitted, see \cite[Lemma 3.4.6]{1811}.
\end{proof}

We are now ready to prove \cite[Proposition 3.4.7]{1811}.
\begin{proposition}[{\cite[Proposition 3.4.7]{1811}}] \label{thm:3.4.7}
  Let a symmetric monoidal $\infty$-category $(\C, \otimes, \boldone)$ be
  given, and let $f \colon B \to A$ be a map of spaces. If $f$ is weakly
  $\C$-ambidextrous and $\otimes$ distributes over $\Delta$-colimits, then the
  the diagram
  \[
    \begin{tikzcd}
    (f \wr C_{p})_{!} \Theta^{p} \arrow[r, "BC_{!}"] \arrow[d, "\Nm_{(f \wr
    C_{p})}"]
    & \Theta^{p} f_{!} \arrow[d, "\Nm_{f}"]
    \\ (f \wr C_{p})_{*} \Theta^{p}
    & f_{*} \Theta^{p} \arrow[l, "BC_{*}"']
    \end{tikzcd}
  \]
  commutes, where $BC_{!}$ and $BC_{*}$ are the Beck-Chevalley transformations
  of left and right adjoints, respectively, of the diagram
  \[
    \begin{tikzcd}
      \Fun(B\wr C_{p}, \C)
      &\Fun(A \wr C_{p}, \C) \arrow[l, "(f\wr C_{p})^{*}"']
      \\\Fun(B, \C)	\arrow[u, "\Theta^{p}"']
      &\Fun(A, \C) \arrow[l, "f^{*}"'] \arrow[u, "\Theta^{p}"']
    \end{tikzcd}.
  \]
\end{proposition}

\begin{proof}
  By \Cref{thm:ff.colimits}, we can assume that $\C$ admits all colimits, so $q
  \colon \LocSys(\C) \to \S$ is a Beck-Chevalley fibration, and $f$ is weakly
  ambidextrous with respect to $q$. By \Cref{thm:theta.pb.preserving}, the
  functor $(-)\wr C_{p} \colon \S \to \S$ preserves pullbacks.  Denote by $p
  \colon \LocSys(\C)_{(-)\wr C_{p}} \to \S$ the base change of $\LocSys(\C) \to
  \S$ along $(-)\wr C_{p}$, so $p$ has the property that the (co)Cartesian
  lifts of $f$ by $p$ are canonically equivalent to the (co)Cartesian lifts of
  the image $f\wr C_{p}$ by $q$. As $f$ is truncated, $f$ is weakly
  ambidextrous with respect to $p$ by \Cref{thm:base.change} (1). As described
  in \Cref{sec:straightening}, by unstraightening,
  the natural transformation $\Theta^{p} \colon \Fun(-, \C) \Longrightarrow
  \Fun((-)\wr C_{p}, \C)$ corresponds to a functor $\Theta^{p} \colon \LocSys(\C)
  \to \LocSys(\C)_{(-) \wr C_{p}}$ which preserves Cartesian edges. By applying
  \Cref{thm:theta.bc} to the diagonal map $\Delta \colon B \to B \times_{A}B$
  induced by $f$, we obtain that the square
  \[
      \begin{tikzcd}
          \Fun(B\wr C_{p}, \C)
          &\Fun( \left(B\times_{A}B \right) \wr C_{p}, \C) \arrow[l, "(\Delta \wr C_{p})^{*}"']
          \\\Fun(B, \C)	\arrow[u, "\Theta^{p}"']
          &\Fun(B\times_{A}B, \C) \arrow[l, "\Delta^{*}"'] \arrow[u, "\Theta^{p}"']
      \end{tikzcd}
  \]
  satisfies the Beck-Chevalley condition.
  The result now follows from \Cref{thm:norm.square}.
\end{proof}

\begin{remark}[]
  As presented in \cite[Proposition 3.4.7]{1811}, the assumptions differ from
  \Cref{thm:3.4.7} in two ways: $(\C,\otimes, \boldone)$ was assumed to be an
  $m$-semiadditively symmetric monoidal $\infty$-category, meaning that all
  $m$-finite maps are weakly $\C$-ambidextrous and that $\otimes$ distributes
  over $m$-finite colimits, and $f\colon B \to A$ was assumed to be an
  $m$-finite map of spaces. As $f$ being $m$-finite implies that the diagonal
  $\Delta$ is $(m-1)$-finite (this follows from \cite[Lemma 5.5.6.15]{htt}),
  $\otimes$ does in fact distribute over $\Delta$-colimits, so
  \cite[Proposition 3.4.7]{1811} follows from \Cref{thm:3.4.7}.
\end{remark}